\documentclass{article}

\pdfoutput=1

\usepackage[utf8]{inputenc}
\usepackage{amssymb, amsmath, amsthm}
\usepackage{hyperref}
\usepackage[title]{appendix}

\usepackage{tikz}
\usepackage{pgfplots}
\usepackage{pgfplotstable}

\newtheorem{lemma}{Lemma}[section]

\newtheorem{proposition}[lemma]{Proposition}

\newtheorem{corollary}[lemma]{Corollary}
\theoremstyle{remark}

\DeclareMathOperator*{\argmin}{argmin}
\DeclareMathOperator{\linspan}{span}
\newcommand{\real}{\mathbb{R}}
\newcommand{\poly}{\mathbb{P}}
\newcommand{\dualp}[1]{\left\langle #1 \right\rangle} 

\newcommand{\param}{\mathcal{P}}
\newcommand{\train}{\mathcal{T}}
\newcommand{\ts}{\Phi}
\newcommand{\err}{\sigma}
\newcommand{\solm}{\mathcal{F}} 

\begin{document}

\title{Transformed snapshot interpolation}
\author{G. Welper\footnote{Department of Mathematics, Texas A\&M University, College Station, Texas 77843-3368, USA, email \href{mailto:welper@math.tamu.edu}{\texttt{welper@math.tamu.edu}}}}
\date{}
\maketitle

\begin{abstract}
  Functions with jumps and kinks typically arising from parameter dependent or stochastic hyperbolic PDEs are notoriously difficult to approximate. If the jump location in physical space is parameter dependent or random, standard approximation techniques like reduced basis methods, PODs, polynomial chaos, etc. are known to yield poor convergence rates. In order to improve these rates, we propose a new approximation scheme. As reduced basis methods, it relies on snapshots for the reconstruction of parameter dependent functions so that it is efficiently applicable in a PDE context. However, we allow a transformation of the physical coordinates before the use of a snapshot in the reconstruction, which allows to realign the moving discontinuities and yields high convergence rates. The transforms are automatically computed by minimizing a training error. In order to show feasibility of this approach it is tested by 1d and 2d numerical experiments.
\end{abstract}



\section{Introduction}

A cornerstone of reduced order modelling, stochastic PDEs and uncertainty quantification, is the efficient approximation of high dimensional PDE solutions $u(x, \mu)$ depending on physical variables $x \in \Omega$ and parametric or random variables $\mu \in \param \subset \real^d$. Many contemporary approximation techniques like e.g. reduced basis methods \cite{RozzaHuynhPatera2008, SenVeroyHuynhEtAl2006, PateraRozza2006}, POD \cite{Sirovich1987, KunischVolkwein2001, KunischVolkwein2002}, Karhunen Lo\`{e}ve expansion \cite{Loeve1978} or polynomial chaos \cite{Wiener1938, XiuKarniadakis2002, Schoutens2000} build upon a reconstruction by a truncated sum 
\begin{equation}
  u(x,\mu) \approx \sum_{i} c_i(\mu) \psi_i(x),
  \label{eq:polyadic-decomposition}
\end{equation}
where the choice and computation of $c_i(\mu)$ and $\psi_i(x)$ depends on the specific method at hand: For reduced basis methods $\psi_i(x) = u(x, \mu_i)$ are snapshots and the $c_i(\mu)$ are computed by a Galerkin projection. For POD and Karhunen Lo\`{e}ve expansions one minimizes the error between $u(x,\mu)$ and any truncated representation of the form \eqref{eq:polyadic-decomposition} and in case of polynomial chaos the functions $c_i(\mu)$ are orthogonal polynomials. Borrowing from the tensor community, we refer to \eqref{eq:polyadic-decomposition} as a \emph{polyadic decomposition} and denote methods based on it by \emph{polyadic decomposition based methods}. The success of all these methods relies on the fact that for many problems one can truncate this sum to a few summands only for the price of a very small error. 

However, this regularity assumption is not always true. An important class of problems are functions $u(x,\mu)$ that have parameter dependent or random jumps or kinks arising e.g. in parametric or stochastic hyperbolic PDEs. Polyadic decomposition based methods are expected to perform poorly for these type of problems. In fact in Appendix \ref{sec:linear-width} we consider a counterexample for which no polyadic decomposition based method can achieve a convergence rate higher than one with respect to the number of summands in a polyadic decomposition. See also \cite{IaccarinoPetterssonNordstrEtAl2010} for a survey in case of uncertainty quantification. 

There are relatively few methods in the literature \cite{ConstantineIaccarino2012, OhlbergerRave2013, JakemanNarayanXiu2013} that directly address this poor performance for parameter dependent jumps and kinks. Instead, much of the work does use polyadic decompositions and focuses on different problems arising in the context of reduced order modelling of parametric hyperbolic PDEs and singularly perturbed problems: Solving the PDE directly in a reduced basis, online/offline decompositions and error estimators, see \cite{HaasdonkOhlberger2008, HaasdonkOhlberger2008a, NguyenRozzaPatera2009, YanoPateraUrban2014, DahmenPleskenWelper2014, Dahmen2015}. 

The goal of this paper is the construction of an alternative approximation method to replace standard polyadic decompositions in order to achieve higher convergence rates for functions $u(x,\mu)$ with parameter dependent jumps and kinks. In addition, the method relies on snapshots and optionally error estimators as input data so that it can be used efficiently and non-intrusively with existing PDE solvers. Somewhat similar to \cite{OhlbergerRave2013}, we allow a transformation $\phi(\mu, \eta): \Omega \to \Omega$ of the physical domain before we use a snapshot $u(x,\eta)$ in the reconstruction of $u(x,\mu)$, i.e.
\begin{equation*}
  u(x,\mu) \approx \sum_{\eta \in \param_n} c_\eta(\mu) u(\phi(\mu, \eta)(x), \eta)
\end{equation*}
where $\eta$ is in some finite set of parameters $\param_n$. The purpose of the additional transform is an alignment of the discontinuities of $u(x,\eta)$ with the ones of $u(x,\mu)$. As a result the discontinuities are ``invisible'' in parameter direction so that very few summands yield accurate approximations. More rigorously, we prove a high order error estimate that does not depend on the regularity of $u(x,\mu)$ itself, but on the regularity of the modified snapshots after alignment which is considerably higher for many practical problems. In addition, because exact alignment is rarely possible in practice, we also take perturbation results into account. Similar to greedy methods for the construction of reduced bases, or neural networks, the transform $\phi$ is computed by minimizing the approximation error on a training sample of snapshots. Although this might seem prohibitively complicated, in Section \ref{sec:optimization} we discuss some preliminary arguments to avoid being trapped in local minima and in Section \ref{sec:numerical-experiments} some 2d numerical experiments are provided where simple subgradient methods provide good results. 

The outlined approximation scheme allows for various realizations with regard to the choices of the coefficients $c_\eta(\mu)$ or the inner transforms $\phi(\mu, \eta)$. Because the main objective of the paper is a proof of principle that one can approximate functions with parameter dependent jumps and kinks with high order from snapshots alone, we usually vote for the simplest possible choices and leave more sophisticated variants for future research.

The paper is organized as follows: In Section \ref{sec:transformed-interpolation} we present the main approximation scheme and a basic error estimate. Then, in Section \ref{sec:stability} we prove stability results and consider approximations of the inner transform. Afterwards we turn to the actual construction of $\phi(\mu,\eta)$. We first discuss some characteristic based approaches and their drawbacks in Section \ref{sec:transform-by-characteristics} and then the optimization of $\phi(\mu,\eta)$ by training errors in Section \ref{sec:optimization}. Finally Section \ref{sec:numerical-experiments} provides some 1d and 2d numerical experiments. For the sake of completeness, in Appendix \ref{sec:linear-width} we consider a counterexample for which no polyadic decomposition based method can achieve high order convergence rates.

\section{Transformed snapshot interpolation}
\label{sec:transformed-interpolation}

In order to motivate the new approximation scheme, we consider the following prototype example throughout this section
\begin{align}
  u(x, \mu) & := \psi \left( \frac{x}{0.4+\mu}-1 \right), &
  \psi(x) & := \left\{ \begin{array}{ll}
    \exp \left( - \frac{1}{1-x^2} \right) & -1 \le x < -\frac{1}{2} \\
    0 & \text{else}
  \end{array} \right.
  \label{eq:mollifier-cut-off}
\end{align}
where $\psi$ is the standard mollifier cut off at $x = -1/2$. This is not necessarily a solution of a PDE, but has the main features we are interested in: a discontinuity that is moving with the parameter. An example for a polyadic decomposition based approximation can be seen in Figure \ref{fig:interpolation} where we recover $u(x,\mu)$ from three snapshots by a polynomial interpolation
\begin{equation*}
  u(x,\mu) \approx \sum_{\eta \in \param_n} \ell_\eta(\mu) u(x, \eta)
\end{equation*}
where $\param_n \subset \param$ are some interpolation points and $\ell_\eta$, $\eta \in \param_n$ are the corresponding Lagrange interpolation polynomials. We see the typical ``staircasing behaviour'' which significantly deteriorates the solutions quality. Although our choice of the polyadic decomposition is perhaps overly simplistic, reduced basis methods and other more sophisticated schemes suffer from the same problem.

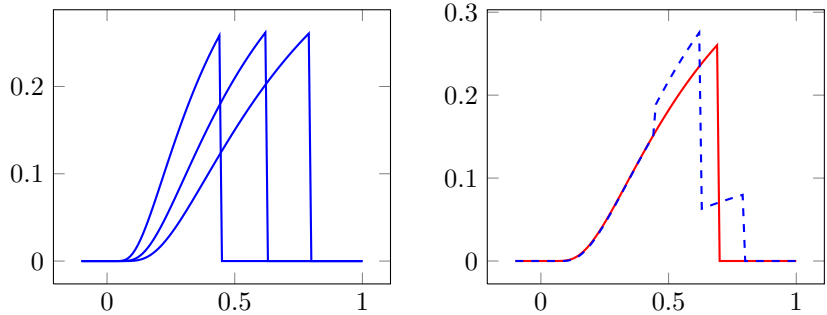
\begin{figure}[htb]

  \hfill
  \begin{tikzpicture}
    \begin{axis}[width=0.5\textwidth]
      \addplot[mark=none, color=blue, thick] table[x=grid points, y=snapshot 0, col sep=comma] {pics/mollifier_cutoff_1d.csv};
      \addplot[mark=none, color=blue, thick] table[x=grid points, y=snapshot 1, col sep=comma] {pics/mollifier_cutoff_1d.csv};
      \addplot[mark=none, color=blue, thick] table[x=grid points, y=snapshot 2, col sep=comma] {pics/mollifier_cutoff_1d.csv};
    \end{axis}
  \end{tikzpicture}
  \hfill
  \begin{tikzpicture}
    \begin{axis}[width=0.5\textwidth]
      \addplot[mark=none, color=red, thick] table[x=grid points, y=truth solution, col sep=comma] {pics/mollifier_cutoff_1d.csv};
      \addplot[mark=none, color=blue, dashed, thick] table[x=grid points, y=interpolation, col sep=comma] {pics/mollifier_cutoff_1d.csv};
    \end{axis}
  \end{tikzpicture}
  \hfill~

  \caption{Left: Snapshots of the parametric function \eqref{eq:mollifier-cut-off} for parameters $\mu = 0.5, 0.85, 1.2$. Right: Exact solution (red) and polynomial interpolation (blue, dashed) for $\mu = 1.0$.}

  \label{fig:interpolation}

\end{figure}

Unlike this superposition of snapshots resulting in the ``staircasing'' phenomena, it seems much more intuitive to compute one snapshot $u(x,\eta)$ and recover the function $u(x,\mu)$ for a different $\mu$ by stretching this snapshot such that the left end of the support is fixed and the jump locations match. To state this intuition in mathematical terms, ``stretching'' one function $u(x,\eta)$ to match a second one $u(x,\mu)$ essentially boils down to a transform $\phi(\mu, \eta): \Omega \to \Omega$ of the physical variables so that we obtain the approximation 
\begin{equation*}
  u(x,\mu) \approx u(\phi(\mu, \eta)(x),\eta).
\end{equation*}
In general, even for optimal choices of $\phi$, we cannot obtain arbitrarily good approximation errors in this way. To obtain convergence, we therefore require in addition that $\eta$ is close to $\mu$ which yields the following simple approximation scheme: First, we choose a finite subset $\param_n \subset \param$ of the parameter domain $\param$ and compute the snapshots $u(x, \eta)$ for $\eta \in \param_n$. Then, given a new $\mu \in \param$, we find the $\eta_\mu \in \param_n$ closest to $\mu$ and approximate
\begin{equation}
  u(x,\mu) \approx u(\phi(\mu, \eta_\mu)(x),\eta_\mu).
  \label{eq:piecewise-constant}
\end{equation}
Besides the snapshots themselves this requires the knowledge of the transforms $(x,\mu) \to \phi(\mu, \eta)(x)$ for finitely many $\eta \in \param_n$. Thus instead of approximating one single function depending on $x$ and $\mu$ we have to find many of them! However, whereas polyadic decompositions perform poorly for $u(x,\mu)$, they often yield good results for the transforms $\phi(\mu, \eta)$: Their smoothness with respect to $\mu$ depends on the smoothness of the jump or kink location with respect to the parameter and not the smoothness of $u(x,\mu)$ itself. For example \eqref{eq:mollifier-cut-off} the jump location is $j(\mu) = \frac{1}{5} + \frac{1}{2} \mu$ so that a linear transform
\begin{equation*}
  \phi(\mu, \eta)(x) = x - j(\mu) + j(\eta)
\end{equation*}
is sufficient to align the jumps. As shown in Figure \ref{fig:transformed-interpolation}, this transform does not align the left end of the supports, but because $\eta$ is close to $\mu$ this is good enough as we see below. For more complicated problems the transform $\phi(\mu, \eta)$ is not explicitly known and we have to find efficient ways to compute it from the given data. We postpone this issue to Section \ref{sec:optimization} and assume for the remainder of this section that $\phi(\mu, \eta)(x)$ is given to us.

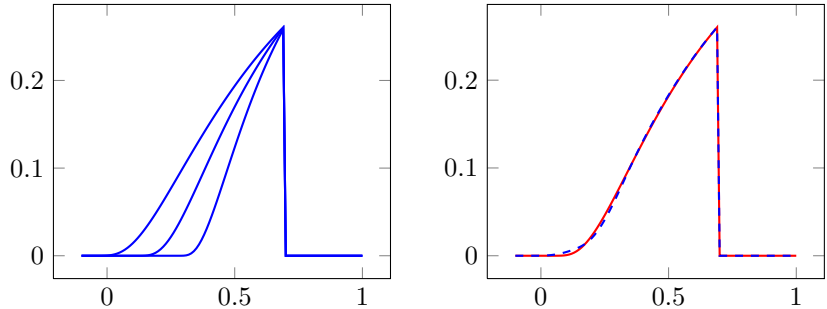
\begin{figure}[htb]

  \hfill
  \begin{tikzpicture}
    \begin{axis}[width=0.5\textwidth]
      \addplot[mark=none, color=blue, thick] table[x=grid points, y=transformed snapshot 0, col sep=comma] {pics/mollifier_cutoff_1d.csv};
      \addplot[mark=none, color=blue, thick] table[x=grid points, y=transformed snapshot 1, col sep=comma] {pics/mollifier_cutoff_1d.csv};
      \addplot[mark=none, color=blue, thick] table[x=grid points, y=transformed snapshot 2, col sep=comma] {pics/mollifier_cutoff_1d.csv};
    \end{axis}
  \end{tikzpicture}
  \hfill
  \begin{tikzpicture}
    \begin{axis}[width=0.5\textwidth]
      \addplot[mark=none, color=red, thick] table[x=grid points, y=truth solution, col sep=comma] {pics/mollifier_cutoff_1d.csv};
      \addplot[mark=none, color=blue, dashed, thick] table[x=grid points, y=transformed interpolation, col sep=comma] {pics/mollifier_cutoff_1d.csv};
    \end{axis}
  \end{tikzpicture}
  \hfill~

  \caption{Left: Transformed snapshots of the parametric function \eqref{eq:mollifier-cut-off} for parameters $\mu = 1.0$ and $\eta = 0.5, 0.85, 1.2$. Right: Exact solution (red) and transformed snapshot interpolation (blue, dashed) for $\mu = 1.0$.}

  \label{fig:transformed-interpolation}

\end{figure}

To assess the approximation error, we observe that our scheme \eqref{eq:piecewise-constant} is a piecewise constant approximation of the \emph{transformed snapshots}
\begin{equation*}
  v_\mu(x,\eta) := u(\phi(\mu, \eta)(x),\eta)
\end{equation*}
with respect to $\eta$ at the point $\eta = \mu$. Thus, we obtain the error estimate
\begin{equation}
  \|v_\mu(x, \mu) - v_\mu(x, \eta_\mu)\|_{L_p} = \mathcal{O}(n^{-1})
  \label{eq:piecewise-const-error}
\end{equation}
where $n$ is the number of snapshots, provided that $v_\mu(x,\eta)$ is differentiable with respect to $\eta$. This is achieved by the inner transform $\phi(\mu, \eta)$: Whereas the original snapshots $u(x,\mu)$ have jumps in parameter direction, the transformed snapshots have jumps in fixed locations independent of $\eta$ resulting in a smooth dependence of $v_\mu(x, \eta)$ on $\eta$. Because $\phi(\mu, \eta)$ is supposed to align the discontinuities and kinks of $u(\cdot, \mu)$ and $u(\cdot, \eta)$, it is natural to require that 
\begin{equation}
  \phi(\mu,\mu)(x) = x
  \label{eq:id}
\end{equation}
which yields $u(x,\mu) = v_\mu(x,\mu)$. With \eqref{eq:piecewise-const-error}, it follows that
\begin{equation*}
  \|u(x, \mu) - v_\mu(x, \eta_\mu)\|_{L_p} = \mathcal{O}(n^{-1})
\end{equation*}
so that our approximation scheme achieves first order convergence.

In comparison, due to lacking smoothness for standard piecewise constant approximations
\begin{equation}
  u(x,\mu) \approx u(x,\eta_\mu)
  \label{eq:piecewise-constant-no-transform}
\end{equation}
we expect convergence rates of $h^{1/p}$ for spacial errors in $L_p$. Thus, depending on the norm, the inner transform yields a gain in the convergence order for $p<1$ or none at all for $p=1$. However, the major impediment is no longer a lack of regularity but the low order convergence of the piecewise constant approximation of $\eta \to v_\mu(x,\eta)$. Therefore, we replace it by a higher order scheme. For simplicity, we confine ourselves to a simple polynomial interpolation and leave more sophisticated choices for future research. Thus for interpolation points $\param_n \subset \param$ and corresponding Lagrange basis polynomials $\ell_\eta$ we define the \emph{transformed snapshot interpolation} by
\begin{equation}
  u(x,\mu) \approx u_n(x,\mu) := \sum_{\eta \in \param_n} \ell_\eta(\mu) u(\phi(\mu, \eta)(x),\eta).
  \label{eq:interpol-transform}
\end{equation}
The input data for this reconstruction is identical to the previous piecewise constant case: We need $|\param_n|$ snapshots and $|\param_n|$ transforms $(x,\mu) \to \phi(\mu,\eta)(x)$, $\eta \in \param_n$. Only the reconstruction formula has been changed to a higher order interpolation. In order to state an error estimate, let $\poly^n$ be the span of the Lagrange basis polynomials and recall that the Lebesgue constant is the norm of the polynomial interpolation operator in the $\sup$-norm which is given by
\begin{equation}
  \Lambda_n := \sup_{\mu \in \param} \sum_{\eta \in \param_n}^n |\ell_\eta(\mu)|.
  \label{eq:lebesgue-constant}
\end{equation}
We obtain the following error estimate.

\begin{proposition}
  \label{prop:outer-error}
  Assume $u_n(x,\mu)$ is defined by the transformed snapshot interpolation \eqref{eq:interpol-transform}. Then for all $\mu \in \param$ the error is bounded by
  \begin{equation*}
    \|u(\cdot, \mu) - u_n(\cdot, \mu)\|_{L_1} \le \Lambda_n \Big\| \inf_{p \in \poly^n} |v(\cdot, \mu) - p| \Big\|_{L_1}. 
  \end{equation*}
\end{proposition}

\begin{proof}
  The proof follows directly from $u(x, \mu) = v_\mu(x,\mu)$ and standard interpolation estimates applied to $\eta \to v_\mu(x,\eta)$.
\end{proof}

For this proposition as well as the remainder of this paper we choose the $L_1$-norm to measure errors because it is the most common choice for hyperbolic PDEs. Also note that the given result is just one option of the various estimates for polynomial interpolation. For example, if we assume analytic dependence of $v_\mu(x,\eta)$ on $\eta$ and use Chebyshev nodes for a one dimensional parameter, we can achieve exponential convergence rates. The most important observation, however, is that the estimate does not involve any regularity assumption of $u(x,\mu)$ itself. Instead it relies on the regularity of $v_\mu(x,\eta)$ with respect to $\eta$ which can be considerably better.

The results for example \eqref{eq:mollifier-cut-off} are shown in the right picture in Figure \ref{fig:transformed-interpolation}. We see a very accurate approximation of the jump, however the approximation quality around the left end of the support of $u(x,\mu)$ is slightly worse than for the original interpolation in Figure \ref{fig:interpolation}. The reason is that the left end of the support is parameter dependent after the transform so that $v_\mu(x,\eta)$ is no longer analytic in $\eta$, however infinitely differentiable. Therefore the loss we suffer at this point is of orders of magnitude less than the staircasing behaviour around the jump of the simple interpolation in Figure \ref{fig:interpolation}.

In summary, instead of approximating the non-smooth function $u(x,\mu)$ directly, for every target $\mu \in \param$ we construct a new smooth function $(x,\eta) \to v_\mu(x, \eta)$ and approximate this function instead. The interpolation condition \eqref{eq:id} guarantees that $u(x,\mu) = v_\mu(x,\eta)$ so that this yields accurate approximations of $u(x,\mu)$ itself as depicted in Figure \ref{fig:solution-manifold}. In addition, for our preliminary simple linear interpolation of $v_\mu(x,\eta)$ this allows an offline/online decomposition: in an offline phase, we compute the snapshots $u(x,\eta)$ as well as the transforms $\phi(\mu, \eta)(x)$ at the interpolation nodes $\eta$ (see Section \ref{sec:optimization} below). Then in an online phase we can efficiently approximate $u(x,\mu)$ for any $\mu \in \param$ by the transformed snapshot interpolation \eqref{eq:interpol-transform}.

\begin{figure}[htb]

  \begin{center}
    \begin{tikzpicture}[scale=0.7]

      \draw[->] (0,0) -- (5,0);
      \draw[->] (0,0) -- (0,5);
      \draw[->] (0,0) -- (-2.5,-2.5);

      \draw[thick, dashed, color=blue] (-2,-1) .. controls (-1,2) and (3,1) .. node[pos=0.5] (A) {} node[below, color=black] {$\mu$} (4,4) node[left] {$v_\mu(\cdot,\eta)$};

      \node[shape=coordinate] (B) at ([shift={(150:1.5)}]A) {x};
      \node[shape=coordinate] (C) at ([shift={(-30:2)}]A) {};

      \draw[thick, color=red] (-4,2) -- (B) -- (C) node[right] {$u(\cdot, \eta)$} -- (3,-1);

    \end{tikzpicture}
  \end{center}

  \caption{The three dimensional coordinate axes indicate the linear space of $x$-dependent functions so that each point of the red line represents a snapshot $x \to u(x,\eta)$ for some parameter $\eta$. The kinks indicate that in our case this solution manifold is not smooth with respect to $\eta$, however note that in reality it is non-smooth for every parameter. The blue dashed line indicates the more smooth transformed snapshots $v_\mu(x, \eta)$.}

  \label{fig:solution-manifold}

\end{figure}
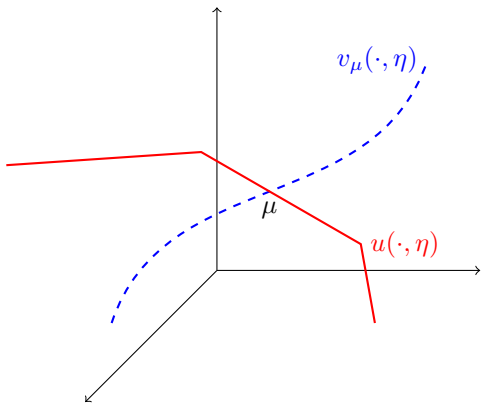

\section{Stability}
\label{sec:stability}

Of course high order smoothness of $v_\mu(x, \eta)$ with respect to $\eta$ needed for high approximation orders in Proposition \ref{prop:outer-error} requires that jumps and kinks are exactly aligned. However, for any finite approximation of the inner transform $\phi$, this is rarely possible. Therefore, we next consider two perturbation results, that allow us to bound the error while taking approximation errors of the inner transform into account. The first one, Lemma \ref{lemma:perturbation-measure}, relies on a measure theoretic argument and allows rather general transforms including ones with kinks as found in e.g. finite element discretizations. The second one, Corollary \ref{cor:perturbation-diffeomorphism}, avoids measure theory, but requires the inner transforms $\phi(\mu, \eta)$ to be diffeomorphisms. 

In the following, let $\varphi(\mu, \eta)(x)$ be a perturbation of $\phi(\mu, \eta)(x)$. To simplify the arguments below, for the time being, we forget about the parameter dependence and consider two transforms $\varphi, \phi: \Omega \to \Omega$ instead. If we assume that each point $\phi(x)$ can be connected to the point $\varphi(x)$ along a curve $\Phi^s(x)$ for $s$ in the interval $[0,1]$, we can rewrite the perturbation by the fundamental theorem of line integrals
\begin{equation*}
  (u \circ \phi)(x) - (u \circ \varphi)(x) = \int_0^1 u'(\Phi^s(x)) \partial_s \Phi^s(x) \, \text{d} t,
\end{equation*}
so that it remains to estimate the right hand side. The map $(x,s) \to \Phi^s(x)$ can be regarded as a function from $\Omega \times [0,1] \to \Omega$ such that 
\begin{align}
  \Phi^0(x) & = \phi(x) & \Phi^1(x) & = \varphi(x).
  \label{eq:homotopy}
\end{align}
which is a homotopy between $\phi$ and $\varphi$ if it is continuous in addition. Furthermore, let $\lambda$ be the Lebesgue measure and $\mathcal{A}$ the Lebesgue $\sigma$-algebra on $\Omega$ and let $\Phi^s_* \lambda$ denote the pushforward measure defined by
\begin{equation*}
  \begin{aligned}
    \Phi^s_* \lambda(A) & = \lambda( (\Phi^s)^{-1}(A)) & & \text{for all } A \in \mathcal{A}. 
  \end{aligned}
\end{equation*}
Then we have the following lemma.

\begin{lemma}
  \label{lemma:perturbation-measure}
  Assume that $u \in BV(\Omega)$ and $\Phi^s$, $0 \le s \le 1$ given by \eqref{eq:homotopy} is measurable and differentiable with respect to $s$ such that
  \begin{equation}
    \begin{aligned}
      \Phi^s_* \lambda(A) & \le c \lambda(A) & & \text{for all } A \in \mathcal{A} \text{ and } 0 \le s \le 1
    \end{aligned}
    \label{eq:pushforward-bound}
  \end{equation}
  and
  \begin{equation}
    \sup_{\substack{0 \le s \le 1 \\ x \in \Omega}} |\partial_s \Phi^s(x)| \le C \|\phi - \varphi\|_{L_\infty(\Omega)}
    \label{eq:length}
  \end{equation}
  for constants $c, C \ge 0$. Then we have
  \begin{equation}
    \|u \circ \phi - u \circ \varphi \|_{L_1(\Omega)} \le c C \|u\|_{BV(\Omega)} \|\phi - \varphi\|_{L_\infty(\Omega)}.
    \label{eq:perturbation-measure}
  \end{equation}
\end{lemma}

Let us discuss the main assumptions before we prove the proposition. If the speed $|\partial_s \Phi^s(x)|$ of each curve $s \to \Phi^s(x)$ is quasi uniform, i.e. equivalent to a constant $S$ for all $x$ and $s$, we have
\begin{equation*}
  \|\partial_s \Phi^s\|_{L_\infty(\Omega \times [0,1])} \sim 
 \int_0^1 | \partial_s \Phi^s(x) | \, \text{d} s  =: l(x)
\end{equation*}
where $l(x)$ is the length of the curve connecting $\phi(x)$ to $\varphi(x)$. In that case assumption \eqref{eq:length} states that, up to a constant, the length of each curve $\Phi^s(x)$ is bounded by the distance $|\phi(x) - \varphi(x)|$ of its endpoints.

In case the domain $\Omega$ is convex, a simple choice of the curves $\Phi^s(x)$ are the convex combinations of the end points:
\begin{equation*}
  \Phi^s(x) = (1-s) \phi(x) + s \varphi(x).
\end{equation*}
In that case, we have $\partial_s \Phi^s(x) = \varphi(x) - \phi(x)$ so that condition \eqref{eq:length} is satisfied.

In order to justify the second assumption \eqref{eq:pushforward-bound}, let us consider the following scenario: Assume that $\phi(x) = x_0 \in \Omega$ and $\varphi(x) = x_1 \in \Omega$ map all of $\Omega$ to single points. Furthermore let $u$ be a piecewise constant function with a jump so that $x_0$ and $x_1$ are on different sides of this jump. On the one hand we obtain $\|u \circ \phi - u \circ \varphi \|_{L_1(\Omega)} = \|u(x_0) - u(x_1)\|_{L_1(\Omega)} = h \lambda(\Omega)$ where $h$ is the hight of the jump. On the other hand we have $\|\phi - \varphi\|_{L_\infty(\Omega)} = |x_0 - x_1|$ which can be made arbitrary small by suitable choices of $x_0$ and $x_1$ on each side of the jump. Thus the main statement \eqref{eq:perturbation-measure} of the proposition is violated. This counterexample relies on the fact that both transforms concentrate all weight in a single point such that $\Phi^i_* \lambda(\{x_i\}) = \lambda(\Omega)$, $i=0,1$ which is ruled out by assumption \eqref{eq:pushforward-bound}.

Finally, we assume that the outer function $u \in BV(\Omega)$ is of bounded variation. This allows jumps and kinks and is one of the most common norms for stability results of hyperbolic PDEs. 

\begin{proof}[Proof of Lemma \ref{lemma:perturbation-measure}]
For the time being, let us assume that $u \in C^1(\Omega)$. Applying the fundamental theorem for line integrals, we obtain
\begin{equation*}
  (u \circ \phi)(x) - (u \circ \varphi)(x) = \int_0^1 u'(\Phi^s(x)) \partial_s \Phi^s(x) \, \text{d} t
\end{equation*}
Thus, we have
\begin{align*}
  \| u \circ \phi - u \circ \varphi \|_{L_1(\Omega)} & = \int_\Omega \left| \int_0^1 u'(\Phi^s(x)) \partial_s \Phi^s(x) \, \text{d} s \right| \text{d} x \\
  & \le \int_\Omega \int_0^1 |u'(\Phi^s(x))| |\partial_s \Phi^s(x)| \, \text{d} s \, \text{d} x \\
  & \le \sup_{\substack{0 \le s \le 1 \\ x \in \Omega}} |\partial_s \Phi^s(x)| \int_\Omega \int_0^1 |u'(\Phi^s(x))| \, \text{d} s \, \text{d} x \\
  & \le C \|\phi - \varphi\|_{L_\infty(\Omega)} \int_0^1 \int_\Omega |u'(\Phi^s(x))| \, \text{d} x \, \text{d} s
\end{align*}
Using the pushforward measure $\Phi^s_* \lambda^{n+1}$ and its bound \eqref{eq:pushforward-bound} we conclude that
\begin{equation}
  \int_\Omega |u'(\Phi^s(x))| \, \text{d} x = \int_\Omega |u'(y)| \, \text{d} \Phi^s_* \lambda(x) \le c \int_\Omega |u'(y)| \, \text{d} x 
  \label{eq:perturbation-measure-proof-1}
\end{equation}
Combining the last two estimates and using that $\int_0^1 \, \text{d} s = 1$ yields
\begin{equation*}
  \| u \circ \phi - u \circ \varphi \|_{L_1(\Omega)} \le c C \|\phi - \varphi\|_{L_\infty(\Omega)} \int_\Omega | u'(y)| \, \text{d} y,
\end{equation*}
which is equivalent to the estimate \eqref{eq:perturbation-measure} we wish to prove.

Finally, we extend the estimate to all $u \in BV(\Omega)$ by using a density argument. To this end note that for all $\epsilon > 0$ there is a $u_\epsilon \in C^1(\Omega)$ such that
\begin{align*}
  \|u - u_\epsilon\|_{L_1(\Omega)} & \le \epsilon & \|u_\epsilon'\|_{L_1(\Omega)} & \le \|u\|_{BV(\Omega)} + \epsilon
\end{align*}
Thus, to apply a density argument, is suffices to bound $\|u \circ \phi - u_\epsilon \circ \phi\|_{L_1(\Omega)}$ and $\|u \circ \varphi - u_\epsilon \circ \varphi\|_{L_1(\Omega)}$. Analogously to \eqref{eq:perturbation-measure-proof-1} we obtain
\begin{align*}
  \|u \circ \phi - u_\epsilon \circ \phi\|_{L_1(\Omega)} & = \int_\Omega |u(\Phi^0(x)) - u_\epsilon(\Phi^0(x))| \, \text{d} x \\
  & = \int_\Omega |u(y) - u_\epsilon(y)| \, \text{d} \Phi^0_* \lambda(y) \\
  & \le c \int_\Omega |u(y) - u_\epsilon(y)| \, \text{d} y \\
  & \le c \|u - u_\epsilon\|_{L_1(\Omega)}
\end{align*}
The bound for $\|u \circ \varphi - u_\epsilon \circ \varphi\|_{L_1(\Omega)}$ follows analogously which completes the proof.
\end{proof}

If the transforms $\Phi^s(x)$ can be chosen to be diffeomorphisms, the pushforward measure is explicitly given by the usual transformation rule
\begin{equation}
  \Phi^s_* \lambda(A) = \int_A |\det D_x (\Phi^s)^{-1}(x)| \, \text{d} \lambda(x)
  \label{eq:pushforward-explicit}
\end{equation}
so that we obtain the following corollary.

\begin{corollary}
  \label{cor:perturbation-diffeomorphism}
  Assume that $u \in BV(\Omega)$ and that $\Phi^s$, $0 \le s \le 1$ given by \eqref{eq:homotopy} are diffeomorphisms for fixed $s$ and differentiable with respect to $s$ such that
  \begin{equation*}
    \begin{aligned}
      |\det D_x (\Phi^s)^{-1}(x)| & \le c & & \text{for all } A \in \mathcal{A} \text{ and } 0 \le s \le 1
    \end{aligned}
  \end{equation*}
  and
  \begin{equation*}
    \sup_{\substack{0 \le s \le 1 \\ x \in \Omega}} |\partial_s \Phi^s(x)| \le C \|\phi - \varphi\|_{L_\infty(\Omega)}
  \end{equation*}
  for constants $c, C \ge 0$. Then we have
  \begin{equation*}
    \|u \circ \phi - u \circ \varphi \|_{L_1(\Omega)} \le c C \|u\|_{BV(\Omega)} \|\phi - \varphi\|_{L_\infty(\Omega)}.
  \end{equation*}
\end{corollary}

\begin{proof}
  We just have to show the bounds \eqref{eq:pushforward-bound} of the pushforward. By its explicit formula \eqref{eq:pushforward-explicit} we have
  \begin{equation*}
    \Phi^s_* \lambda(A)  = \int_A |\det D_x (\Phi^s)^{-1}(x)| \, \text{d} \lambda(x) \le c \lambda(A)
  \end{equation*}
  for all $A \in \mathcal{A}$ and $0 \le s \le 1$ so that the corollary follows from Lemma \ref{lemma:perturbation-measure}
  
\end{proof}

Let us now consider the transformed snapshot interpolation \eqref{eq:interpol-transform} again. Assume that there is a transform $\phi(\mu, \eta)(x)$ that aligns the jumps and kinks exactly so that we obtain high convergence rates in Proposition \ref{prop:outer-error}. In general, we have to find a finite approximation to this exact transform, say $\phi_m(\mu, \eta)(x)$. Note that according to \eqref{eq:interpol-transform} we only need to know this function for the $|\param_n|$ nodes $\eta \in \param_n$, so that we have to approximate $|\param_n|$ functions depending on $x \in \Omega$ and a parameter $\mu \in \param$. Of course this is exactly the same problem as approximating a function $u(x,\mu)$ which is our initial problem, however, the regularity of $\phi(\mu,\eta)(x)$ can be much more favorable as we have seen in the introduction in Section \ref{sec:transformed-interpolation} or as we will see in Section \ref{sec:transform-by-characteristics} below. Therefore, we can apply a more classical polyadic decomposition based approach to find an approximation $\phi_m(\mu, \eta)(x)$ of the inner transform. Replacing the exact transform by the approximate one in the transformed snapshot interpolation  yields
\begin{equation}
  u(x,\mu) \approx u_{n,m}(x,\mu) := \sum_{\eta \in \param_n} \ell_\eta(\mu) u(\phi_m(\mu, \eta)(x),\eta).
  \label{eq:interpol-inner-outer}
\end{equation}
Combining the error estimate of Proposition \ref{prop:outer-error} with the perturbation result Lemma \ref{lemma:perturbation-measure} we arrive at the following Proposition. 

\begin{proposition}
  \label{prop:inner-outer-error}
  Assume that $u \in BV(\Omega)$ and that there are curves $\Phi^s(\mu, \eta)(x)$, $0 \le s \le 1$ for $x \in \Omega$, $\mu \in \param$, $\eta \in \param_n$ measurable and differentiable with respect to $s$ such that
  \begin{align*}
    \Phi(\mu,\eta)^0(x) & = \phi(\mu, \eta)(x) & \Phi(\mu,\eta)^1(x) & = \phi_m(\mu, \eta)(x).
  \end{align*}
  and
  \begin{equation*}
    \begin{aligned}
      \Phi(\mu,\eta)^s_* \lambda(A) & \le c \lambda(A) & & \text{for all } A \in \mathcal{A} \text{ and } 0 \le s \le 1
    \end{aligned}
  \end{equation*}
  and
  \begin{equation*}
    \sup_{\substack{0 \le s \le 1 \\ x \in \Omega}} |\partial_s \Phi(\mu, \eta)^s(x)| \le C \|\phi(\mu, \eta) - \phi_m(\mu, \eta)\|_{L_\infty(\Omega)}
  \end{equation*}
  for constants $c, C \ge 0$. Furthermore let $u_{n,m}(x,\mu)$ be defined by the transformed snapshot interpolation \eqref{eq:interpol-inner-outer}. Then for all $\mu \in \param$ we have the error estimate
  \begin{align*}
    \|u(\cdot, \mu) - u_{n,m}(\cdot, \mu)\|_{L_1} & \le \Lambda_n \Big\| \inf_{p \in \poly^n} |v(\cdot, \mu) - p| \Big\|_{L_1} \\ 
    & \quad + c C \Lambda_n \max_{\eta \in \param_n} \|u(\cdot, \eta)\|_{BV(\Omega)} \max_{\eta \in \param_n} \| \phi(\mu, \eta) - \phi_m(\mu, \eta) \|_{L_\infty(\Omega)},
  \end{align*}
  where $\Lambda_n$ is the Lebesgue constant \eqref{eq:lebesgue-constant}.
\end{proposition}

\begin{proof}
We have
\begin{equation*}
  \|u(\cdot, \mu) - u_{n,m}(\cdot, \mu)\|_{L_1(\Omega)} \le
  \|u(\cdot, \mu) - u_n(\cdot, \mu)\|_{L_1(\Omega)} +
  \|u_n(\cdot, \mu) - u_{n,m}(\cdot, \mu)\|_{L_1(\Omega)}
\end{equation*}
With Proposition \ref{prop:outer-error} the first term can be estimated by
\begin{equation*}
  \|u(\cdot, \mu) - u_n(\cdot, \mu)\|_{L_1} \le \Lambda_n \Big\| \inf_{p \in \poly^n} |v(\cdot, \mu) - p| \Big\|_{L_1}.
\end{equation*}
In order to estimate the second term, using the definition of the Lebesgue constant and Lemma \ref{lemma:perturbation-measure}, we obtain
\begin{align*}
  \|u_n(\cdot, \mu) - u_{n,m}(\cdot, \mu)\|_{L_1(\Omega)} & \le \sum_{\eta \in \param_n} |\ell_\eta(\mu)| \Big\| u(\phi(\mu, \eta)(x),\eta) - u(\phi_m(\mu, \eta)(x),\eta) \Big\|_{L_1(\Omega)} \\
  & \le \Lambda_n \max_{\eta \in \param_n} \| u(\phi(\mu, \eta)(x),\eta) - u(\phi_m(\mu, \eta)(x),\eta) \|_{L_1(\Omega)} \\
  & \le c C \Lambda_n \max_{\eta \in \param_n} \|u(\cdot, \eta)\|_{BV(\Omega)} \max_{\eta \in \param_n} \| \phi(\mu, \eta) - \phi_m(\mu, \eta) \|_{L_\infty(\Omega)}
\end{align*}
Combining all three estimates completes the proof.
\end{proof}

If the $\mu$ dependence of $\phi(\mu, \eta)$ is smooth, we can use a polyadic decomposition for its approximation. Although there are much more sophisticated methods, possibly the simplest choice is a linear interpolation
\begin{equation}
  \phi_m(\mu, \eta)(x) = \sum_{\nu \in \hat{\param}_m} \hat{\ell}_{\nu}(\mu) \phi(\nu,\eta)(x)
  \label{eq:transform-interpolation}
\end{equation}
where $\hat{\ell}_\nu$ are Lagrange basis polynomials with respect to nodes in some finite set $\hat{\param}_m$. With this inner approximation, the error bound of Proposition \ref{prop:inner-outer-error} depends of the smoothness of the transformed snapshot $v_\mu(x, \eta)$ with respect to $\eta$ and of the transforms $(x,\mu) \to \phi(\mu, \eta)(x)$ with respect to $\mu$. If both dependencies are analytic, for suitable interpolation points the error decays exponentially.

\section{Inner transforms by characteristics}
\label{sec:transform-by-characteristics}

We still have to choose an inner transform $\phi(\mu, \eta)$ such that the transformed snapshots $v_\mu(x, \eta)$ are as smooth in $\eta$ as possible. One obvious idea that comes to mind is to somehow make use of characteristics. In this section, we discuss some problems that arise from that approach for the Riemann problem for Burgers' equation. In our example, the parameter is the hight of the jump in the initial condition which yields the parametric PDE
\begin{equation*}
  \begin{aligned}
    u_t + \left(\frac{1}{2} u^2 \right)_x & = 0 & \text{in } \real \times \real^+ \\
    u(x,0) = g_\mu(x) & = \left\{ \begin{array}{ll}
      \mu & x \le 0 \\
      0 & x > 0
    \end{array} \right. & \text{for } t = 0.
  \end{aligned}
\end{equation*}
In addition, we assume that $\mu > 0$ so that the solution has a shock along the curve $\frac{1}{2} \mu t$. To write down an explicit solution formula, let $\chi(x,t)$ be the origin (at time $t=0$) of the characteristic passing through the point $(x,t)$. It is easily seen to be
\begin{equation}
  \chi_\mu(x,t) = \left\{ \begin{array}{ll}
    x - \mu t & x \le \frac{1}{2} \mu t \\
    x         & x >   \frac{1}{2} \mu t.
  \end{array} \right.
  \label{eq:backward-characteristic} 
\end{equation}
Because $u(x,t,\mu)$ is constant along characteristics, we obtain 
\begin{equation}
  u(x, t, \mu) = g_\mu(\chi_\mu(x,t)).
  \label{eq:burgers-solution}
\end{equation}
The previous discussion aside, a simple idea for an approximation scheme is to encode or approximate $g_\mu(x)$ and the characteristics $\chi_\mu(x,t)$ and then use the exact solution formula \eqref{eq:burgers-solution} to reconstruct $u(x,t,\mu)$. In spirit this is similar to our original idea \eqref{eq:piecewise-constant} where the snapshots $u(\cdot, \eta)$ are replaced by $g_\eta(\cdot)$ and the transform $\phi(\mu, \eta)$ by the characteristic $\chi_\mu$. However, from the explicit formula \eqref{eq:backward-characteristic} for $\chi_\mu$, we see that this function has a parameter dependent jump. Thus, in general, we have to face the same difficulties for approximating the parameter dependent characteristic $(x,t,\mu) \to \chi_\mu(x,t)$ as for the original solution $(x,t,\mu) \to u(x,t,\mu)$ so that there is no progress with respect to this issue.

If we want to use characteristics to define the inner transform $\phi(\mu, \eta)$ of the transformed snapshot interpolation, the problems are even more complicated. As for \eqref{eq:burgers-solution} we can follow the characteristics backward in time, but because we to not evaluate the initial condition $g_\mu$ but a snapshot $u(x,t,\eta)$, we then follow the characteristics forward in time with a different parameter. To this end, let $\varphi_\mu(y,t)$ be the position of the characteristic at time $t$, starting at the initial position $y$ at time $t=0$:
\begin{equation}
  \varphi_\mu(y,t) = \left\{ \begin{array}{ll}
    y + \mu t          & y \le - \frac{1}{2} \mu t \\
    \frac{1}{2} \mu t  & - \frac{1}{2} \mu t \le y < \frac{1}{2} \mu t \\
    y                  & \frac{1}{2} \mu t < y.
  \end{array} \right.
  \label{eq:forward-characteristic} 
\end{equation}
Then we can transform one solution for parameter $\eta$ into a solution for parameter $\mu$ by
\begin{equation}
  u(x, t, \mu) = g_\mu(\chi_\mu(x,t)) = \frac{\mu}{\eta} g_\eta(\chi_\mu(x,t)) = \frac{\mu}{\eta} u(\varphi_\eta(\chi_\mu(x,t),t), t, \eta).
  \label{eq:solution-forward-backward}
\end{equation}
However, this formula is only correct for $\mu \ge \eta$. The reason is that the interval of points at $t=0$ that eventually end up in the shock at time $t$ is strictly larger for larger parameters. Thus, for $\mu < \eta$ there is a interval $I$ around the shock location of $\mu$ for which $\chi_\mu(I,t)$ is mapped into the shock location $\frac{1}{2} \eta$ by the forward characteristic $\varphi_\eta$. Thus, the right hand side of \eqref{eq:solution-forward-backward} has only one single value in the interval $I$ or is undefined whereas the left hand side has to different values and the formula is thus not correct.

Nonetheless, for $\mu \ge \eta$, we can define the transform 
\begin{equation}
  \phi(\mu, \eta)(x,t) = \varphi_\eta(\chi_\mu(x,t),t)
  \label{eq:transform-characteristics}
\end{equation}
so that by \eqref{eq:solution-forward-backward} the transformed snapshot is
\begin{equation*}
  v_\mu(x, t, \eta) = u(\phi(\mu, \eta)(x,t), \eta) = u(\varphi_\eta(\chi_\mu(x,t),t), t, \eta) = \frac{\eta}{\mu} u(x,t,\mu)
\end{equation*}
which is clearly smooth and in fact even linear in $\eta$. Using $\mu \ge \eta$ it is easy to verify that $\phi(\mu, \eta)$ is
\begin{align*}
  \phi(\mu, \eta)(x,t) & = \left\{ \begin{array}{ll}
    x - (\mu-\eta)t & x \le \frac{1}{2} \mu t \\
    x & x > \frac{1}{2} \mu t
  \end{array} \right.
\end{align*}
Note that for our approximation scheme \eqref{eq:interpol-transform} we need to know the function $(x,\mu) \to \phi(\mu, \eta)(x)$ for finitely many $\eta \in \param$. Again, this function has a $\mu$ dependent jump so that its approximation poses the same difficulties already encountered for $u(x,t,\mu)$ itself.

However, we are not obliged to use the transform \eqref{eq:transform-characteristics} based on characteristics. By noting that the shock location is $\frac{1}{2} \mu t$, simply shifting the whole solution in $x$-direction by
\begin{equation}
  \phi(\mu,\eta)(x,t) = x + \frac{1}{2} (\mu - \eta)t
  \label{eq:shift}
\end{equation}
aligns the shocks, i.e. the transformed snapshot $u(\phi(\mu,\eta)(x,t),t,\eta)$ has its shock in the location $\frac{1}{2} \mu t$ which is the shock location for parameter $\mu$. In addition the interpolation condition \eqref{eq:id} is obviously satisfied. Note that by \eqref{eq:backward-characteristic} and \eqref{eq:burgers-solution} the parametric solution is
\begin{equation*}
  u(x,t,\mu) = \left\{ \begin{array}{ll}
    \mu & x \le \frac{1}{2} \mu \\
    0   & x >   \frac{1}{2} \mu.
  \end{array} \right.
\end{equation*}
so that the transformed snapshot becomes
\begin{equation*}
  v_\mu(x,t,\eta) = \left\{ \begin{array}{ll}
    \eta & x \le \frac{1}{2} \mu \\
    0   & x >   \frac{1}{2} \mu.
  \end{array} \right. = \frac{\eta}{\mu} u(x,t,\mu)
\end{equation*}
which is the same as for the characteristics based transform, but now for all $\mu, \eta \in \param$ . Recall that the error estimate of Proposition \ref{prop:outer-error} just requires smoothness with respect to $\eta$ which is obviously the case. However, opposed to the characteristic construction also $\phi(\mu,\eta)(x)$ is smooth in $\mu$ so that polyadic decomposition based methods yield accurate approximations of $\phi$ at low cost.

\section{Optimizing the interpolation error}
\label{sec:optimization}

\subsection{Generalized gradients}
\label{sec:generalized-gradients}

We still need to find a realistic way to actually compute the inner transform $\phi(\mu, \eta)$. Similar to the construction of reduced bases, PODs or neural networks, we aim at finding an inner transform $\phi$ that minimizes the approximation error. To this end, we measure the error in the $sup$-norm with respect to the parameter which is typical for reduced basis methods but not mandatory. It follows that the overall error is given by
\begin{equation}
  \err_\param(\phi) := \sup_{\mu \in \param} \err_\mu(\phi),
  \label{eq:sup-error}
\end{equation}
where 
\begin{equation*}
  \err_\mu(\phi) := \|u(\cdot, \mu) - u_n(\cdot, \mu; \phi)\|_{L_1(\Omega)}
\end{equation*}
is the error for one fixed parameter. To make the dependence on the inner transform more explicit, in this section we denote the transformed snapshot interpolation \eqref{eq:interpol-transform} by $u_n(x, \mu; \phi) = u_n(x, \mu)$. In practice it is not possible to minimize the error $\err_\param(\phi)$ directly because it would require the knowledge of all functions $u(\cdot, \mu)$ for all parameters $\mu \in \param$. To this end, we only assume to know the errors $\err_\mu(\phi)$ of a finite training sample $\mu \in \train \subset \param$ so that the overall error \eqref{eq:sup-error} is replaced by the training error
\begin{equation}
  \err_\train(\phi) := \sup_{\mu \in \train} \err_\mu(\phi).
  \label{eq:training-error}
\end{equation}
Although surrogates for the training error are available for some singularly perturbed problems \cite{NguyenRozzaPatera2009, YanoPateraUrban2014, DahmenPleskenWelper2014, Dahmen2015} we omit these in favor of future research. Instead, we resort to an explicit knowledge of some training snapshots $u(\cdot, \mu)$, $\mu \in \train$ in addition to the snapshots that are used for the reconstruction \eqref{eq:interpol-transform} itself. In contrast to the reduced basis method this severely limits the size of the training sample. Nevertheless, in Section \ref{sec:numerical-experiments} we consider examples which yield good results with roughly twice as many training snapshots than reconstruction snapshots, so that the additional burden of the training samples is reasonable.

Because we are explicitly interested in non-smooth functions $u$, the error $\err_\train(\phi)$ is a non-trivial objective function to minimize. The next proposition shows that despite possible jumps of $u$ the error is Lipschitz continuous, nonetheless. The assumptions of this proposition are essentially the same as for Lemma \ref{lemma:perturbation-measure} and are commented right after it.

\begin{proposition}
  \label{prop:error-lipschitz}
  Assume that $u \in BV(\Omega)$ and that there are curves $\Phi^s(\mu, \eta)(x)$, $0 \le s \le 1$ for $x \in \Omega$, $\mu \in \param$, $\eta \in \param_n$ measurable and differentiable with respect to $s$ such that
  \begin{align}
    \Phi^0(\mu, \eta)(x) & = \phi(\mu, \eta)(x) & \Phi^1(\mu, \eta)(x) & = \varphi(\mu, \eta)(x).
    \label{eq:homotopy-transform}
  \end{align}
  and
  \begin{equation*}
    \begin{aligned}
      \Phi^s(\mu, \eta)_* \lambda(A) & \le c \lambda(A) & & \text{for all } A \in \mathcal{A} \text{ and } 0 \le s \le 1
    \end{aligned}
  \end{equation*}
  and
  \begin{equation*}
    \sup_{\substack{0 \le s \le 1 \\ x \in \Omega}} |\partial_s \Phi(\mu,\eta)^s(x)| \le C \|\phi(\mu, \eta) - \varphi(\mu, \eta)\|_{L_\infty(\Omega)}
  \end{equation*}
  for constants $c, C \ge 0$. Then we have
  \begin{equation}
    |\err_\train(\phi) - \err_\train(\varphi)| \le cC \Lambda_n \sup_{\eta \in \param_n} \|u(\cdot, \eta)\|_{BV(\Omega)} \sup_{\substack{\mu \in \param \\ \eta \in \param_n}} \left\| \phi(\mu, \eta) - \varphi(\mu, \eta)\right\|_{L_\infty(\Omega)}.
    \label{eq:error-lipschitz}
  \end{equation}
\end{proposition}

\begin{proof}
  Note that the triangle inequality implies that
  \begin{equation}
    |\err_\train(\phi) - \err_\train(\varphi)| = \left|\sup_{\mu \in \train} \err_\mu(\phi) - \sup_{\mu \in \train} \err_\mu(\varphi)\right| \le \sup_{\mu \in \train} |\err_\mu(\phi) - \err_\mu(\varphi)|
    \label{eq:error-lipschitz-1}
  \end{equation}
  so that it is sufficient to bound $|\err_\mu(\phi) - \err_\mu(\varphi)|$. To this end note that
  \begin{align*}
    |\err_\mu(\phi) - \err_\mu(\varphi)| & = \Big| \|u(\cdot, \mu) - u_n(\cdot, \mu; \phi)\|_{L_1(\Omega)} - \|u(\cdot, \mu) - u_n(\cdot, \mu; \varphi)\|_{L_1(\Omega)} \Big| \\
    & \le  \Big\| \big[ u(\cdot, \mu) - u_n(\cdot, \mu; \phi) \big] - \big[ u(\cdot, \mu) - u_n(\cdot, \mu; \varphi) \big] \Big\|_{L_1(\Omega)} \\
    & \le  \left\| \sum_{\eta \in \param_n} \ell_\eta(\mu) \big[u(\phi(\mu, \eta)(x),\eta) - u(\varphi(\mu, \eta)(x),\eta) \big] \right\|_{L_1(\Omega)} \\
    & \le  \sum_{\eta \in \param_n} |\ell_\eta(\mu)| \left\| u(\phi(\mu, \eta)(x),\eta) - u(\varphi(\mu, \eta)(x),\eta) \right\|_{L_1(\Omega)}
  \end{align*}
  Due to the given assumptions, we can now apply Lemma \ref{lemma:perturbation-measure} to conclude that
  \begin{align*}
    |\err_\mu(\phi) - \err_\mu(\varphi)| & \le cC \sum_{\eta \in \param_n} |\ell_\eta(\mu)| \|u(\cdot, \eta)\|_{BV(\Omega)} \left\| \phi(\mu, \eta) - \varphi(\mu, \eta)\right\|_{L_\infty(\Omega)} \\
    & \le cC \Lambda_n \sup_{\eta \in \param_n} \|u(\cdot, \eta)\|_{BV(\Omega)} \sup_{\eta \in \param_n} \left\| \phi(\mu, \eta) - \varphi(\mu, \eta)\right\|_{L_\infty(\Omega)}.
  \end{align*}
  With \eqref{eq:error-lipschitz-1} this yields claimed estimate \eqref{eq:error-lipschitz}.
  
\end{proof}

In order to optimize the training error $\err_\train(\phi)$, we search for a minimizer in a set of candidate transforms $\phi \in \ts \subset X$ in a Banach space $X$. The space of continuous functions $C(\param \times \param \times \Omega)$ with the additional restrictions from Proposition \ref{prop:error-lipschitz} seems to be a reasonable choice for $\ts$ because in the last proposition the transform error is measured in the supremum norm. In general this objective function is not differentiable so that we cannot rely on standard gradient based optimizers. However, because $\err_\train(\phi)$ is Lipschitz continuous according to the last proposition, we can use optimization methods from non-smooth optimization \cite{Kiwiel1985, BurkeLewisOverton2005} relying on the generalized Clarke gradient \cite{Clarke2013}. To this end, for a direction $v \in X$, we first define the generalized directional derivative
\begin{equation*}
  \err^\circ(\phi; v) := \limsup_{\varphi \to \phi; \, h \downarrow 0} \frac{\err(\varphi + hv) - \err(\varphi)}{h},
\end{equation*}
where we suppress the additional $\train$ subscript of $\sigma$ for simplicity. Note that this limit is well defined because $\err$ is Lipschitz continuous. In order to define a gradient from these directional derivatives, recall that in the differentiable case one can define the gradient $\nabla \err \in X^*$ variationally by
\begin{equation*}
  \begin{aligned}
    \partial_v \err & = \dualp{\nabla \err, v}, & & \text{for all } v \in X
  \end{aligned}
\end{equation*}
where $X^*$ is the dual space of $X$ and $\dualp{\cdot, \cdot}$ the corresponding dual pairing. Likewise, in the Lipschitz continuous case we define the generalized gradient by
\begin{equation*}
  \begin{aligned}
    \partial_C \err = \{g \in X^* | \, \err^\circ(\phi; v) \ge \dualp{g, v}, \text{ for all } v \in X \}.
  \end{aligned}
\end{equation*}
In case $\err$ is differentiable this reduces to the standard gradient and in case $\err$ is convex to the subgradient. 

In the literature on non-smooth optimization one can find several algorithms to minimize $\err_\train(\phi)$ based on this generalized gradient. For some first numerical tests, we use a simple subgradient method:
\begin{align}
  \phi^{k+1} & = \phi^{k} + h_k \Delta^k & \Delta^k & \in \partial_C \err(\phi^k) & \err^0 = Id
  \label{eq:subgrad}
\end{align}
where $Id(x) = x$ is the identity transform. Note that this method does not use the full generalized gradient $\partial_C \err$ but just one element of it in each step. This is typical for non-smooth/convex optimization methods because usually the full generalized gradient is not known. Since non-smooth optimization problems often have kinks at the minimum itself, we cannot use standard techniques to control the step size and use a fixed rule
\begin{equation}
  \begin{aligned}
    h_k & = \alpha k^\beta, & \alpha & > 0, & 0 & < \beta < 1 
  \end{aligned}
  \label{eq:step-size}
\end{equation}
instead. This simple method converges for convex functions \cite{Bertsekas1999, Kiwiel1985} and yields good results in the numerical experiments below. More sophisticated methods including convergence analysis for non-convex problems are available, see e.g. \cite{Kiwiel1985, BurkeLewisOverton2005}.

\subsection{Global minima?}

Because the objective function $\err_\train(\phi)$ is non-convex in general, we must make sure that we do not end up in a suboptimal local minimum. In this section, we discuss some preliminary ideas to overcome this issue for a simple class of 1d problems. To this end, let us consider piecewise constant functions
\begin{equation}
  \begin{aligned}
    u(x,\mu) & = \left\{ \begin{array}{ll}
      u_0, & \text{for } x < x_1(\mu) \\
      u_i, & \text{for } x_i(\mu) \le x < x_{i+1}(\mu) \\
      u_n, & \text{for } x_n(\mu) \le x
    \end{array} \right.
  \end{aligned}
  \label{eq:piecewise-const}
\end{equation}
with smooth parameter dependent jump locations $x_i(\mu)$. We assume that the order $x_0(\mu) < \dots < x_n(\mu)$ never changes and that the jump locations are well separated i.e. there is a constant $L \ge 0$ with $|x_i(\mu) - x_{i-1}(\mu)| \ge L$, $i=1, \dots, n$.

\paragraph{Transformed snapshot interpolation and training error}

Let us first state the transformed snapshot interpolation for these functions and the optimization problem to find the inner transform. Because $u(x,\mu)$ is piecewise constant in $x$, for transforms $\phi(\mu, \eta)$ that perfectly align the discontinuities the transformed snapshots $v_\mu(x, \eta)$ are constant in $\eta$. Therefore, it is sufficient to confine ourselves to one single snapshot for the outer interpolation, say $u(x,\mu_0)$ with node $\param_n = \{\mu_0\}$ so that we obtain
\begin{equation}
  u(x,\mu) = v_\mu(x,\mu_0) = u(\phi(\mu, \mu_0), \mu_0).
  \label{eq:tsi-constant-u}
\end{equation}
It follows that we have to compute inner transforms $\phi(\mu, \mu_0)$ for all $\mu \in \param$ and the single fixed node $\mu_0$. To this end, we assume to know additional training snapshots $u(\cdot, \mu_1), \dots, u(\cdot, \mu_m)$ for interpolation points $\mu_1 < \dots < \mu_m$, with $\mu_0 \le \mu_1$ for simplicity. Because we just use on snapshot for the outer interpolation the training error \eqref{eq:training-error} reduces to 
\begin{equation}
  \err_\train(\phi) = \sup_{1 \le i \le m} \err_{\mu_i}(\phi) = \sup_{1 \le i \le m} \|u(\cdot, \mu_i) - u(\phi(\mu_i, \mu_0), \mu_0)\|_{L_1(\Omega)}.
\end{equation}
Since all transforms $\phi(\mu_i, \mu_0)$, $1 \le 1 \le m$ are uncorrelated, we can further simplify this and optimize for each transform $\phi(\mu_i, \mu_0)$ individually, which yields
\begin{equation}
  \phi(\mu_i, \mu_0) = \argmin_{\varphi} \|u(\cdot, \mu_i) - u(\varphi(\cdot), \mu_0)\|_{L_1(\Omega)}
  \label{eq:opt-i}
\end{equation}
for $i=1, \dots, n$. Finally, with $\phi(\mu_0, \mu_0)(x) = x$ to ensure the interpolation condition \eqref{eq:id}, as in \eqref{eq:transform-interpolation} we can define the full transform by an interpolation
\begin{equation*}
  \phi_m(\mu, \mu_0) = \sum_{i=0}^m \hat{\ell}(\mu) \phi(\mu_i, \mu_0)
\end{equation*}
where $\hat{\ell}_i$ are the Lagrange polynomials for the nodes $\mu_0, \dots, \mu_m$.

\paragraph{Counterexample: Local minima}

It remains to solve the $m$ optimization problems \eqref{eq:opt-i}. Already for this simple problem, optimization methods relying on local search for updates can easily be fooled into a non-optimal local minimum. To this end, consider the example in Figure \ref{fig:example-local} defined by
\begin{align}
  u(x,\mu) & = \chi_{I_1(\mu)} + \chi_{I_2(\mu)}, & I_1(\mu) & = [\mu, \mu+1), & I_2(\mu) & = [\mu+4, \mu+5),
  \label{eq:two-boxes}
\end{align}
where $\chi$ is the characteristic function and the parameter shifts the entire function. 

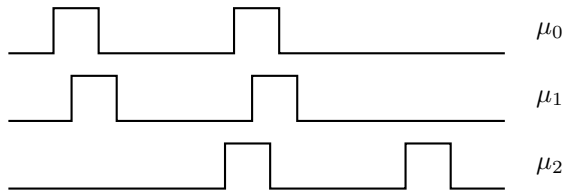
\begin{figure}[htb]

  \begin{center}
    \begin{tikzpicture}[scale=0.6]
    
      \newcommand{\xmin}{-1}
      \newcommand{\xmax}{10}
      \newcommand{\twoboxes}[3]{
      \draw[#3, thick] (\xmin,0+#2) -- (#1,0+#2) -- (#1,1+#2) -- (1+#1,1+#2) -- (1+#1,0+#2) -- (4+#1,0+#2) -- (4+#1,1+#2) -- (5+#1,1+#2) -- (5+#1,0+#2) -- (\xmax,0+#2) ;
      }
    
      \twoboxes{0}{0}{black}
      \twoboxes{0.4}{-1.5}{black}
      \twoboxes{3.8}{-3}{black}
    
      \node () at (1+\xmax,0.5) {$\mu_0$};
      \node () at (1+\xmax,-1) {$\mu_1$};
      \node () at (1+\xmax,-2.5) {$\mu_2$};
      
    \end{tikzpicture}
  \end{center}

  \caption{Functions \eqref{eq:two-boxes} for various parameters}
  \label{fig:example-local}

\end{figure}

The snapshots $\mu_0$ and $\mu_2$ in Figure \eqref{fig:example-local} are arranged such that the first interval of $u(x,\mu_2)$ intersects the second one of $u(x,\mu_0)$. Therefore the error $\err_{\mu_2}(\phi)$ is simply the area of the mismatch between the overlapping intervals plus the area of the two mismatched intervals. Note in particular that any small perturbation of $\phi(\mu_2, \mu_0)$ does not change the later error contribution. It follows that optimization schemes exclusively relying on local information are fooled into a local minimum that matches the (wrong) intersecting intervals.

However, the situation changes, when the difference between the snapshot parameter $\mu_0$ and the training parameter $\mu$ is small as e.g. for $\mu_1$ in Figure \ref{fig:example-local}. In this case there are no mismatched intervals and intuitively already simple subgradient based optimization schemes converge to the correct global minimum perfectly aligning the two functions. That this is in fact true is discussed in the following.

\paragraph{Local convexity}

We confine ourselves to spatially monotone transforms $\phi(\mu, \mu_0)$ in agreement with our assumption that the order of the jumps $x_i(\mu)$ does not change. Ideally, we search for transforms which exactly match the jump locations
\begin{equation*}
  \phi(\mu, \mu_0)(x_i(\mu)) = x_i(\mu_0) \Leftrightarrow \phi(\mu,\mu_0)^{-1}(x_i(\mu_0)) = x_i(\mu).
\end{equation*}
Practically, we have to deal with perturbations so that the jumps only approximately match
\begin{equation*}
  \phi(\mu,\mu_0)^{-1}(x_i(\mu_0)) \approx x_i(\mu).
\end{equation*}
If this matching error is sufficiently small compared to the minimal jump distance $L$, such that only adjacent intervals of $u(\cdot, \mu)$ and $v_\mu(x,\mu_0)$ overlap the training error simplifies to
\begin{equation}
  \err_\mu(\phi) = \sum_{i=1}^n |u_i - u_{i-1}| |x_i(\mu) - \phi(\mu,\mu_0)^{-1}(x_i(\mu_0))|.
  \label{eq:err-piecewise-const}
\end{equation}
This error is convex in $\phi(\mu, \mu_0)^{-1}$ which is not surprising because we assume that we are already close to a minimum. Whereas in principle the convexity allows us to compute optimal transforms this is not yet very practical since it requires very good initial values. However the identity transform $id(x) = x$ satisfied 
\begin{equation}
  |x_i(\mu) - id^{-1}(x_i(\mu_0))| = |x_i(\mu) - x_i(\mu_0)|
  \label{eq:err-id}
\end{equation}
which is sufficiently small to guarantee the error representation \eqref{eq:err-piecewise-const} for $\mu$ sufficiently close to $\mu_0$ so that $id$ can be used as an initial value in that case. 

Formally, in order to obtain a convex optimization problem, we augment the original training error minimization \eqref{eq:opt-i} with the following constraints
\begin{equation}
  \begin{aligned}
    & \err_\mu(\phi) \to \min \\
    & \phi(\mu, \mu_0)^{-1} \in C(\Omega) \text{ strictly monotonically increasing} \\
    & |x_i(\mu_0) - \phi(\mu,\mu_0)^{-1}(x_i(\mu_0))| \le B, \quad i=0, \dots, n
  \end{aligned}
\end{equation}
for some constant $B>0$. Both constraints are clearly convex in $\phi(\mu, \mu_0)^{-1}$. To show convexity of the objective function, note that by the triangle inequality we have
\begin{equation*}
  |x_i(\mu) - \phi(\mu,\mu_0)^{-1}(x_i(\mu_0))| \le |x_i(\mu) - x_i(\mu_0)| + |x_i(\mu_0) - \phi(\mu,\mu_0)^{-1}(x_i(\mu_0))|.
\end{equation*}
Thus, for $\mu$ sufficiently close to $\mu_0$ and $B$ sufficiently small the objective function reduces to \eqref{eq:err-piecewise-const} which is convex with respect to $\phi(\mu, \mu_0)^{-1}$. According to \eqref{eq:err-id} the identity is allowed by the constraints so that we know a suitable initial value for iterative optimization methods. Moreover for a transform $\hat{\phi}(\mu, \mu_0)$ perfectly aligning the discontinuities, we have
\begin{equation*}
  |x_i(\mu_0) - \hat{\phi}(\mu,\mu_0)^{-1}(x_i(\mu_0))| = |x_i(\mu_0) - x_i(\mu)| 
\end{equation*}
which is also allowed by the constraints. It follows that the optimal error is $\err_\mu(\hat{\phi}) = 0$ which is therefore a global minimum.

\paragraph{Locality by transitivity}

In summary for $\mu$ sufficiently close to $\mu_0$, we can reliably find a global minimum of the error $\err_\mu(\phi)$ by solving a convex optimization problem, eventually with the identity $id(x) = x$ as initial value. But what about larger differences of $\mu$ and $\mu_0$ as e.g. in our counter example with $\mu_2$ in Figure \ref{fig:example-local}? To this end, recall that the main purpose of the transform $\phi(\mu, \eta)$ is the alignment of jumps and kinks. Thus, if $x(\mu)$ is the location of a jump for parameter $\mu$, we want the transforms to satisfy
\begin{equation*}
  \phi(\mu, \eta)(x(\mu)) = x(\eta).
\end{equation*}
This condition guarantees that the transformed snapshot $v_\mu(x(\mu), \eta) = u(x(\eta), \mu)$ has a jump at $x(\mu)$ just as the target function $u(x,\mu)$. But this alignment condition is transitive in nature: For three consecutive parameters $\mu_0$, $\mu_1$ and $\mu_2$ we have
\begin{equation*}
  \big( \phi(\mu_1, \mu_0) \circ  \phi(\mu_2, \mu_1)\big) (x(\mu_2)) = x(\mu_0)
\end{equation*}
so that $\phi(\mu_1, \mu_0) \circ  \phi(\mu_2, \mu_1)$ correctly aligns the jumps for parameters $\mu_0$ and $\mu_2$. Because this alignment property is a major requirement for the transform $\phi(\mu_2, \mu_0)$, we can define it that way
\begin{equation}
  \phi(\mu_2, \mu_1) := \phi(\mu_1, \mu_0) \circ  \phi(\mu_2, \mu_1).
  \label{eq:compose}
\end{equation}
Let us apply this construction to our example transformed snapshot interpolation \eqref{eq:tsi-constant-u} where we have one snapshot at $\mu_0$ and $m$ training snapshots at $\mu_1 \le \dots \le \mu_m$. If we enforce transitivity \eqref{eq:compose}, we define
\begin{equation*}
  \phi(\mu_i, \mu_0) := \phi(\mu_{1}, \mu_0) \circ \dots \circ \phi(\mu_i, \mu_{i-1})
\end{equation*}
so that we are left with the calculation of the ``local in $\mu$'' transforms $\phi(\mu_i, \mu_{i-1})$. Because they are supposed to align jumps and kinks of $u(\cdot, \mu_i)$ and $u(\cdot, \mu_{i-1})$ we can compute them by solving the optimization problem 
\begin{equation}
  \phi(\mu_i, \mu_{i-1}) = \argmin_{\varphi} \|u(\cdot, \mu_i) - u(\varphi(\cdot), \mu_{i-1})\|_{L_1(\Omega)}
  \label{eq:convex-opt}
\end{equation}
which is the same as the original problem \eqref{eq:opt-i} with $\mu_0$ replaced by $\mu_{i-1}$. By choosing sufficiently many training snapshots, we can enforce $|\mu_i - \mu_{i-1}|$ to be sufficiently small such that the optimization problem \eqref{eq:convex-opt} becomes convex. Therefore, we can reliably find global minimizers $\phi(\mu_i, \mu_{i-1})$ and in turn by \eqref{eq:compose} a transform that perfectly aligns the discontinuities for the parameters $\mu_0$ and $\mu_i$. This leads to a zero training error $\err_\train(\phi)$ which is therefore a global minimum.

In summary, for the reconstruction of piecewise constant functions in 1d from one snapshot, we can find $\phi(\mu, \eta)$ as the global minimum of the training error provided there are sufficiently many training snapshots. Of course the argument relies on a couple of assumptions that are not true in more general cases. Notably, the functions $u(x,\mu)$ might not be piecewise constant, we eventually want to use more snapshots for the outer interpolation and the parameter and spacial dimensions can be larger that one. Nonetheless, a transitivity property of the transforms is still realistic. As for the simple example of this section, this allows some locality in the parameter for the minimization of the training error. How to make use of this and to what extend this is helpful for more complicated scenarios is an open problem.

\section{Numerical experiments}
\label{sec:numerical-experiments}

In this section, we consider some first numerical tests of the transformed snapshot interpolation \eqref{eq:interpol-transform}. First, in Section \ref{sec:gaussian}, a 1d example is presented, where the focus is on the approximation rate, while the inner transforms $\phi(\mu, \eta)$ are given explicitly. Then, in Section \ref{sec:burgers} the method is tested with a 2d Burgers Riemann problem where the solution is explicitly known. Finally, in Section \ref{sec:shock-bubble} the method is applied to a shock bubble interaction which is a more challenging test case for the optimizer of the inner transform.

\subsection{Cut off Gaussian}
\label{sec:gaussian}

For a first numerical experiment, we consider the parametric function
\begin{align}
  u(x, \mu) & := N \left( \frac{x}{0.4+\mu}-1 \right), &
  N(x) & := \left\{ \begin{array}{ll}
    0.4 \, e^{-7.0 \, x^2} & -1 \le x < -\frac{1}{2} \\
    0 & \text{else}
  \end{array} \right.
  \label{eq:normal-cut-off}
\end{align}
which is a scaled and shifted Gaussian, cut off at a parameter dependent location, see Figure \ref{fig:example-gaussian}. This function is not chosen with a parametric PDE in mind, but has a parameter dependent jump and because it is known explicitly it is well suited for analyzing the performance of the transformed snapshot interpolation. For the snapshots we consider to alternatives: First we use the exact function \eqref{eq:normal-cut-off} and second we interpolate it by piecewise linear functions on a uniform grid. The latter choice should simulate the outcome of PDE solvers which yield similar approximations of the parametric solution. Due to the simplicity of the example, we choose shifts for the inner transform:
\begin{equation*}
  \phi(\mu, \eta)(x) = x + s(\mu, \eta).
\end{equation*}
Recall that for the transformed snapshot interpolation \eqref{eq:interpol-inner-outer}, we only need to know $s(\mu, \eta)$ for interpolation points $\mu \in \hat{\param}_m$ and $\eta \in \param_n$. Thus, we can encode $\phi$ by storing $m \times n$ floating point numbers. For all examples, we choose $\param_n = \hat{\param}_n$. In addition, in this example we only consider the approximation properties of the transformed snapshot interpolation. In order not to interfere with the optimizer for actually finding the transform, we use an explicit formula for $s(\mu, \eta)$ that exactly aligns the jumps and consider the optimizer in the numerical examples below.

The numerical results are summarized in Figure \ref{fig:example-gaussian}. In case we use exact snapshots, we see a more than polynomial convergence rate. Note that after aligning the snapshots, the transformed snapshots $v_\mu(x,\eta)$ are analytic in $\eta$ so that this behaviour is in line with the error bounds of Proposition \ref{prop:outer-error}. However, for the linearly interpolated snapshots the situation is different. The error first decays and then saturates at a level dependent on the spacial grid resolution. These levels correspond to the maximal error of the snapshots themselves as shown in Table \ref{table:example-gaussian}. This makes sense because the transformed snapshot interpolation error can hardly be better than the error of the snapshots it relies on. 

For a comparison, Figure \ref{fig:example-gaussian} also contains the error of a simple polynomial interpolation without transform. We see the typical staircasing behaviour and an error that is orders of magnitudes worse than the one with previous transform.

\begin{figure}[htb]

  \hfill
  \begin{tikzpicture}
    \begin{loglogaxis}[width=0.5\textwidth]
      \addplot[mark=*] table[x=n parameters, y=Example 0, col sep=comma] {pics/example_shift_normal_grid_errors.csv};
      \addplot[mark=*] table[x=n parameters, y=Example 1, col sep=comma] {pics/example_shift_normal_grid_errors.csv};
      \addplot[mark=*] table[x=n parameters, y=Example 2, col sep=comma] {pics/example_shift_normal_grid_errors.csv};
      \addplot[mark=*] table[x=n parameters, y=Example 3, col sep=comma] {pics/example_shift_normal_grid_errors.csv};
      \addplot[mark=+, color=red] table[x=n parameters, y=Example 4, col sep=comma] {pics/example_shift_normal_grid_errors.csv};
      \addplot[mark=x, color=blue] table[x=n parameters, y=id, col sep=comma] {pics/example_shift_normal_grid_errors.csv};
    \end{loglogaxis}
  \end{tikzpicture}
  \hfill
  \begin{tikzpicture}
    \begin{axis}[width=0.5\textwidth]
      \addplot[mark=none, color=black, thick] table[x=grid points, y=truth solution, col sep=comma] {pics/example_shift_normal_grid_plots.csv};
      \addplot[mark=none, color=blue, densely dotted, thick] table[x=grid points, y=interpolation, col sep=comma] {pics/example_shift_normal_grid_plots.csv};
      \addplot[mark=none, color=red, dashed, thick] table[x=grid points, y=transformed solution, col sep=comma] {pics/example_shift_normal_grid_plots.csv};
    \end{axis}
  \end{tikzpicture}
  \hfill~

  \caption{Left: Errors of transformed snapshot interpolations of \eqref{eq:normal-cut-off}. Black with dot marks: Snapshots are piecewise linear interpolations with uniform grid sizes $h=0.1, 0.01, 0.001, 0.0001$. Red with $+$ marks: Snapshots are exact functions without any approximation. Blue with $\times$ marks: Interpolation without transform. Right: Truth solution (black) and interpolation from nine snapshots with (red, dashed) and without (blue dotted) transform.}

  \label{fig:example-gaussian}

\end{figure}
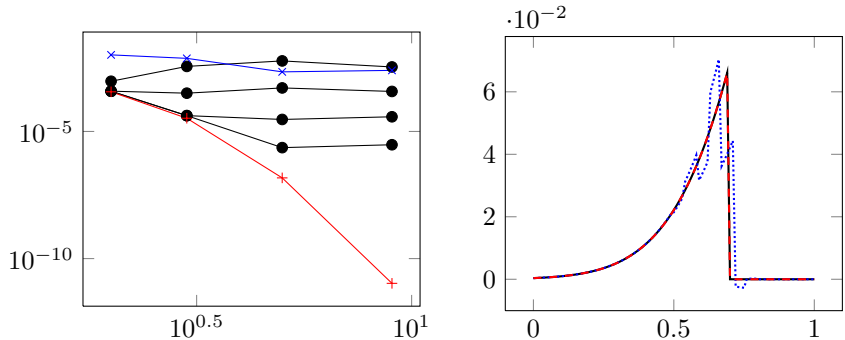

\begin{table}[htb]
\begin{filecontents*}{pics/example_shift_normal_grid_errors.csv}
n parameters, Example 0, Example 1, Example 2, Example 3, Example 4, id
2, 0.000939535108511, 0.000382628786143, 0.000381243261194, 0.000381089020551, 0.000361498063538, 0.0102283877635
3, 0.00360267256937, 0.000319285578613, 4.18081157511e-05, 4.17408245651e-05, 3.30158399316e-05, 0.00746764420676
5, 0.00592438230155, 0.000513449812331, 2.958513986e-05, 2.30753963625e-06, 1.4926573007e-07, 0.00219956334709
9, 0.00343041163278, 0.00037359070568, 3.76552914264e-05, 2.9953611588e-06, 1.0662284044e-11, 0.00251217409116
\end{filecontents*}
\pgfplotstabletypeset[
    every head row/.style={
        before row={
            \hline
            \multicolumn{1}{|c|}{} & 
            \multicolumn{5}{|c|}{transformed snapshot interpolation} & 
            \multicolumn{1}{c|}{interpolation} \\
            \hline
        },
        after row={\hline}
    },
    col sep=comma,
    columns/{n parameters}/.style={column name=$n$, column type/.add={|}{}},
    columns/{Example 0}/.style={column name=$0.1$, column type/.add={|}{}},
    columns/{Example 1}/.style={column name=$0.01$, column type/.add={|}{}},
    columns/{Example 2}/.style={column name=$0.001$, column type/.add={|}{}},
    columns/{Example 3}/.style={column name=$0.0001$, column type/.add={|}{}},
    columns/{Example 4}/.style={column name=exact, column type/.add={|}{}},
    columns/{id}/.style={column name=exact, column type/.add={|}{|}},
    every last row/.style={after row={ 
      \hline
      \hline
      \multicolumn{1}{|c|}{} & 
      \multicolumn{5}{|c|}{maximal $L_1(\Omega)$ error of the snapshots} &
      \multicolumn{1}{|c|}{} \\
      \hline
      & $3.58 \cdot 10^{-3}$ & $3.53 \cdot 10^{-4}$ & $3.48 \cdot 10^{-5}$ & $3.66 \cdot 10^{-10}$ & & \\
      \hline
    }}
]{pics/example_shift_normal_grid_errors.csv}
\caption{Errors for example \eqref{eq:normal-cut-off} for different uniform spacial grids and number of snapshots ($n$). The last line contains the maximal $L_1(\Omega)$ error of the respective snapshots.}
\label{table:example-gaussian}
\end{table}

\subsection{2d Burgers' equation}
\label{sec:burgers}

For a second example, we consider the two dimensional Burgers' equation
\begin{equation*}
  \partial_t u + \nabla \cdot \left(\frac{1}{2} u^2 \boldsymbol{v} \right) = 0
\end{equation*}
with $\boldsymbol{v} = (1,1)^T$ and the initial condition
\begin{equation*}
  u(x, y, 0) = \left\{ \begin{array}{rl}
    -0.2 & \text{if } x<0.5 \text { and } y>0.5 \\
    -1.0 & \text{if } x>0.5 \text { and } y>0.5 \\
     0.5 & \text{if } x<0.5 \text { and } y<0.5 \\
     0.8 & \text{if } x>0.5 \text { and } y<0.5
  \end{array} \right.
\end{equation*}
on the unit cube $\Omega = [0,1]^2$. According to \cite{GuermondPasquettiPopov2011, GerhardMuller2014} the exact solution for this problem is
\begin{equation}
  u(x, y, t) = \left\{ \begin{array}{rlcl}
    \begin{array}{rr}
      -0.2 \\ 0.5
    \end{array} & \text{if } 
      x < \frac{1}{2} - \frac{3t}{5} & 
    \text{and} & \left\{\begin{array}{l}
      y > \frac{1}{2} + \frac{3t}{20}, \\ \text{otherwise},
    \end{array} \right.
    \\
    \begin{array}{rr}
      -1.0 \\ 0.5
    \end{array} & \text{if } 
      \frac{1}{2} - \frac{3t}{5} < x < \frac{1}{2} - \frac{t}{4} & 
    \text{and} & \left\{\begin{array}{l}
      y > -\frac{8x}{7} + \frac{15}{14} - \frac{15t}{28}, \\ \text{otherwise},
    \end{array} \right.
    \\
    \begin{array}{rr}
      -1.0 \\ 0.5
    \end{array} & \text{if } 
      \frac{1}{2} - \frac{t}{4} < x < \frac{1}{2} + \frac{t}{2} & 
    \text{and} & \left\{\begin{array}{l}
      y > \frac{x}{6} + \frac{5}{12} - \frac{5t}{24}, \\ \text{otherwise},
    \end{array} \right.
    \\
    \begin{array}{rr}
      -1.0 \\ \frac{2x - 1}{2t}
    \end{array} & \text{if } 
      \frac{1}{2} + \frac{t}{2} < x < \frac{1}{2} + \frac{4t}{5} & 
    \text{and} & \left\{\begin{array}{l}
      y > x - \frac{5}{18t} \left( x + t - \frac{1}{2} \right)^2, \\ \text{otherwise},
    \end{array} \right.
    \\
    \begin{array}{rr}
      -1.0 \\ 0.8
    \end{array} & \text{if } 
      \frac{1}{2} + \frac{4t}{5} < x & 
    \text{and} & \left\{\begin{array}{l}
      y > \frac{1}{2} - \frac{t}{10}, \\ \text{otherwise},
    \end{array} \right.
    \\
  \end{array} \right.
  \label{eq:burgers-exact}
\end{equation}
For a simple test of the transformed snapshot interpolation, we consider the time $t$ as the parameter of interest so that the snapshots are solutions at various time instances used for the reconstruction of the solution at intermediate times. Note that with this choice of the parameter the solution has exactly the features in question: it has parameter dependent jumps and kinks along non-trivial curves. In addition the exact solution is known which is helpful for an exact assessment of the numerical errors. In order to simulate a numerical solution of the Burger's equation, we sample the snapshots on a $100 \times 100$ grid and use a piecewise linear reconstruction from these samples. Also the integrals for evaluating the errors during the optimization of the inner transform $\phi$ rely on this grid.

Figure \ref{fig:burgers} shows the results for a reconstruction at time $0.45$ from two snapshots at times $0.3$ and $0.5$ with an additional training snapshot at time $0.4$ to define the training error $\err_\train(\phi)$. For a first test, the inner transforms $\phi(\mu, \eta)$ for $\mu, \eta \in \{0.3, 0.5\}$ are simply polynomials mapping $\real^2 \to \real^2$. In general this choice does not guarantee that $\Omega$ is mapped to itself, however it is easy to enforce that the edges of the rectangular domain are mapped to itself so that small perturbations of the identity are diffeomorphisms. In order to be able to align both the kink and the jump in the lower right corner, for the $x$-component we choose a (multivariate) polynomial of degree $3 \times 2$ and for the $y$-component a polynomial of degree $2 \times 2$. For the optimization of the training error with respect $\phi$ we use a subgradient method \eqref{eq:subgrad}, \eqref{eq:step-size} with $500$ steps of the rather conservative fixed step size
\begin{equation*}
  h_k =  \frac{10^{-3}}{(i+1)^{0.1}}.
\end{equation*}
Compared to a classical polynomial interpolation, the additional transform almost completely removes the artificial staircasing behaviour. Also the kinks around the ``ramp'' in the upper left corner of the figures are much better resolved. The $L_1$ errors computed by an adaptive quadrature instead of the grid of the snapshots are as follows.
\begin{center}
  \begin{tabular}{ll}
    $L_1$-error interpolation                           &  0.0355373675439  \\
    $L_1$-error transformed snapshot interpolation      &  0.00739513000396 \\
    maximal snapshot $L_1$-error                        &  0.0051148730424
  \end{tabular}
\end{center}
In conclusion the additional inner transform reduces the error almost by a factor of five compared to a plain polynomial interpolation. Note that the error of the transformed snapshot interpolation is almost down to the maximal error of the snapshots themselves. As we have verified in Figure \ref{fig:example-gaussian} for the 1d example of Section \ref{sec:gaussian} we expect the error to saturate somewhere around this level so that more snapshots or degrees of freedom for the inner transform are not expected to yield major improvements. This is a serious bottleneck for computing convergence rates: Due to the jump discontinuities the maximal error of the snapshots converges with a low rate. The resulting high number of spacial degrees of freedom renders the computation of convergences rates challenging.

\begin{figure}[htb]

  \hfill
  \includegraphics[width=0.25\textwidth]{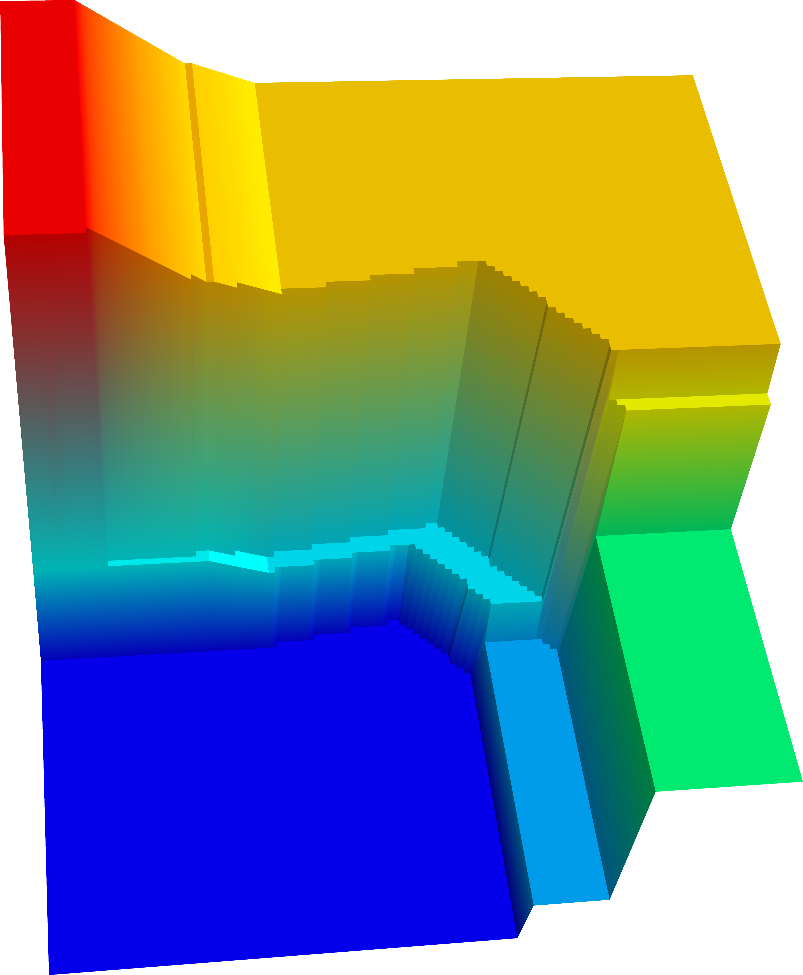}
  \hfill
  \includegraphics[width=0.25\textwidth]{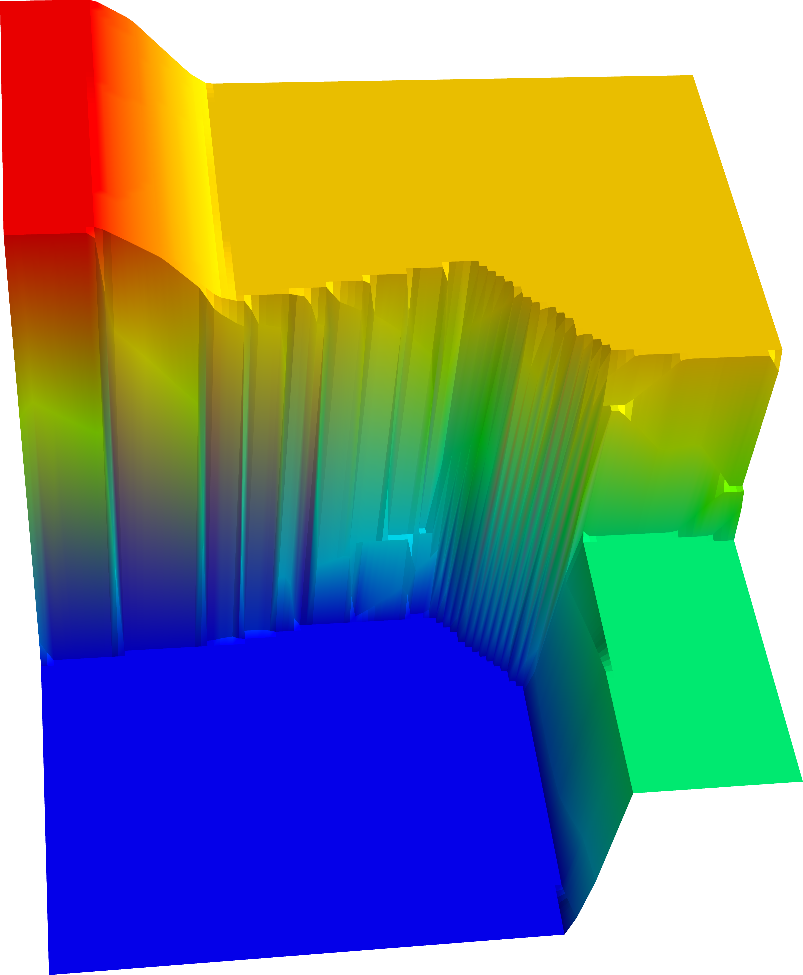}
  \hfill
  \includegraphics[width=0.25\textwidth]{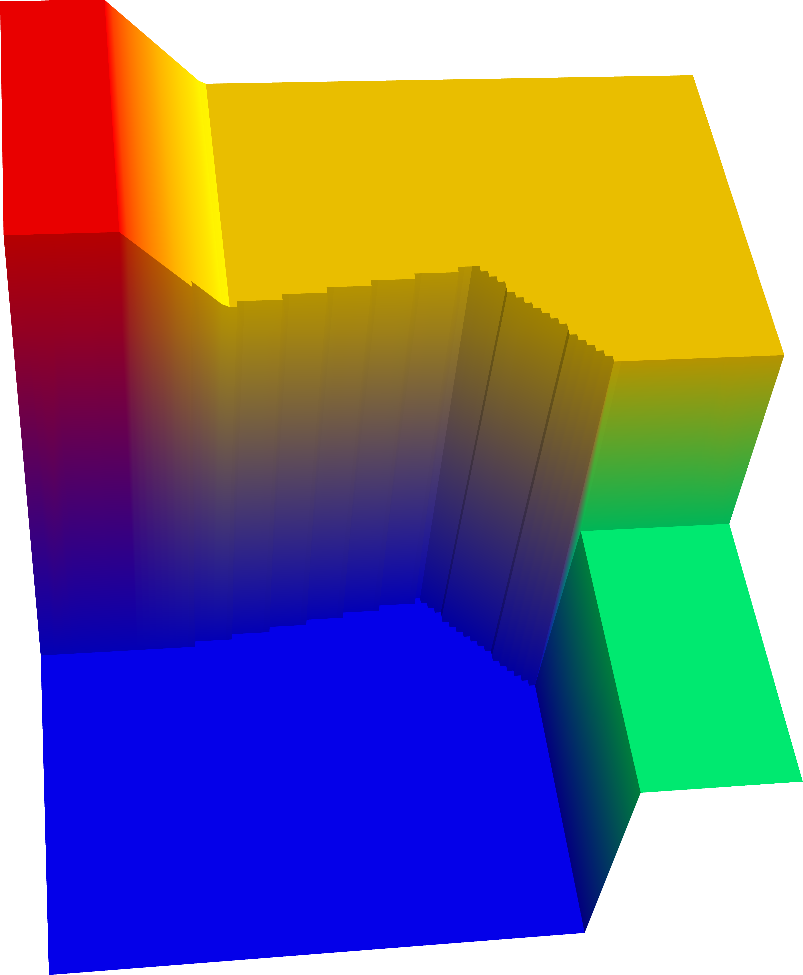}
  \hfill~

  \caption{Reconstruction of the solution \eqref{eq:burgers-exact} for Burgers equation at the time $t=0.45$. Left: polynomial interpolation, middle: transformed snapshot interpolation, right: exact solution.}
  \label{fig:burgers}
\end{figure}

\subsection{Shock bubble interaction}
\label{sec:shock-bubble}

For a last more complicated example, we consider a compressible Euler simulation of a shock-bubble interaction \cite{Nazarov2015}. Because the code for the above examples relies on piecewise linear interpolation on a uniform grid to represent the snapshots it is straight forward to read them from pictures. For the shock bubble interaction experiments they are frames from a video showing the time evolution of the density provided by \cite{Nazarov2015a, Nazarov2015}. Figure \ref{fig:bubble} shows the snapshots and the reconstruction at a new time by linear interpolation and transformed snapshot interpolation. Comparable to Section \ref{sec:burgers}, we simply choose third order polynomials for the inner transform $\phi$ mapping the edges of the domain to itself. We see that the linear interpolation result basically shows the two bubbles from the original snapshots, whereas the true solution of course just has one bubble. Using the additional transform $\phi$, the second picture finds the correct location of the shock and the bubble. Thus despite lots of more complicated fine structure in the pictures, the optimizer reliably finds the correct transform. Having a closer look, the reconstructed bubble appears to be a little blurred. However, we only use third order polynomials which eventually is insufficient for a perfect alignment of the shapes.

\begin{figure}[htb]

  \includegraphics[width=0.45\textwidth]{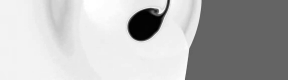}
  \hfill
  \includegraphics[width=0.45\textwidth]{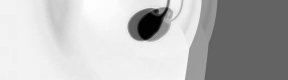}
  \hfill~\\[0.2cm]
  \includegraphics[width=0.45\textwidth]{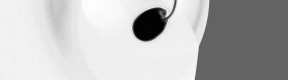}
  \hfill
  \includegraphics[width=0.45\textwidth]{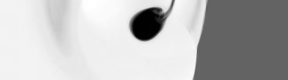}
  \hfill~\\[0.2cm]
  \includegraphics[width=0.45\textwidth]{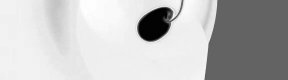}

  \caption{Left column: Snapshots for time indices $3.5$, $3.75$ and $4.0$ where the middle one is only used for the optimization of the transform $\phi$. Right column: reconstruction at time index $3.7$ by linear interpolation (top) and transformed snapshot interpolation (bottom).}
  \label{fig:bubble}
\end{figure}

\begin{appendices}

\section{Linear Width}
\label{sec:linear-width}

As an example for the limitations of polyadic decomposition based methods, let us consider their best possible performance for the following simple parametric transport problem:
\begin{equation*}
  \begin{aligned}
    A_\mu u_\mu & := u_t + \mu u_x = 0 & & \text{for } 0<t<1, \, x \in \real \\
    u(x,0)  & = g(x) := \left\{ \begin{array}{ll} 
    0 & x<0 \\
    1 & x \ge 0
  \end{array} \right. & & \text{for } t=0,
  \end{aligned}
\end{equation*}
with parameter $\mu \in \param = [\mu_{\min}, \mu_{\max}] \subset \real$. Its solution is given by
\begin{equation}
  u(x,t) = g(x-\mu t).
  \label{eq:exact-solution}
\end{equation}
The typical benchmark for the performance of reduced basis methods is the Kolmogorov $n$-width 
\begin{equation*}
  d_n(\solm) = \inf_{\dim Y = n} \sup_{u \in \solm} \inf_{\phi \in Y} \|u - \phi\|_{L_1}
\end{equation*}
of the solution manifold
\begin{equation}
  \solm := \{u(\cdot, \mu) | \mu \in \param \} = \{g(x - \mu t) | \mu \in \param \}.
  \label{eq:solution-transport}
\end{equation}
However, with $X_n := \linspan \{\psi_i : \, i=1, \dots, n\}$ based on the $\psi_i$ of the polyadic decomposition \eqref{eq:polyadic-decomposition}, we conclude that
\begin{equation*}
  d_n(\solm) \le \sup_{u \in \solm} \inf_{\phi \in X_n} \|u(\cdot, \mu) - \phi\|_{L_1} \le \sup_{u \in \solm} \left\| u(\cdot, \mu) - \sum_{i=1}^n c_i(\mu) \psi_i \right\|_{L_1}.
\end{equation*}
Therefore, the errors of polyadic decomposition based methods, including but not restricted to reduced basis methods, are lower bounded by the Kolmogorov $n$-width if the error is measured in the $\sup$-norm with respect to the parameter variable. Of course this measure for the error is not appropriate for all applications, but we use it here to exemplify the limitations of the standard polyadic decomposition based methods.

For our simple model problem the Kolmogorov $n$-width is bounded according to the following Proposition. The proof is similar to \cite{Donoho2001}, see also \cite{LorentzGolitschekMakovoz1996}. Note that order one is already achieved by a simple nonadaptive piecewise constant approximation as e.g. in \eqref{eq:piecewise-constant-no-transform}.

\begin{proposition}
  The Kolmogorov width of the solution manifold $\solm$ defined in \eqref{eq:solution-transport} satisfies
  \begin{equation*}
    d_n(\solm) \sim n^{-1},
  \end{equation*}
  i.e. $d_n(\solm)$ is equivalent to $n^{-1}$ up to a constant.
\end{proposition}

\begin{proof}

We show the lower bound by comparing the Kolmogorov $n$-width of the solution manifold $\solm$ to the known width of a ball in the $L_1$-norm. For the construction of this ball, we choose a uniform distribution of $k+1$ snapshots with parameters
\begin{equation*}
  \begin{aligned}
    \mu_i & = \mu_{\min} + \frac{i}{k}(\mu_{\max} - \mu_{\min}), & i & = 0, \dots, k,
  \end{aligned}
\end{equation*}
where $k$ will be chosen below. Note that these specific snapshots are only used for this proof and are not necessarily used in actual approximation methods like e.g. reduced basis methods. We use the snapshots to define the functions
\begin{equation*}
  \begin{aligned}
    \xi_i & = u_{\mu_i} - u_{\mu_{i-1}}, & i & = 1, \dots, k,
  \end{aligned}
\end{equation*}
which will be the corners of the $L_1$-ball. From 
\begin{equation*}
  \inf_{y \in Y} \|\xi_i - y\|_{L_1} \le \inf_{y \in Y} \|u_{\mu_i} - y\|_{L_1} + \inf_{y \in Y} \|u_{\mu_{i-i}} - y\|_{L_1} \le 2 \sup_{u \in \solm} \inf_{y \in Y} \|u - y\|_{L_1}
\end{equation*}
for any linear space $Y$ of dimension at most $n$ follows that
\begin{equation*}
  d_n(\{ \xi_0, \dots, \xi_k \}) \le 2 d_n(\solm).
\end{equation*}
We complete the set $\{ \xi_0, \dots, \xi_k \}$ to a full ball without increasing the Kolmogorov width. To this end assume that $\lambda_i$, $i=1, \dots, k$ satisfy $\sum_{i=1}^k |\lambda_1| \le 1$ and $y_i$, $i=1, \dots, k$ are the minimizers of $\inf_{y \in Y} \|\xi_i - y\|_{L_1}$. Then we have
\begin{multline*}
  \inf_{y \in Y} \left \| \sum_{i=1}^k \lambda_i \xi_i - y \right\|_{L_1} \le \left \| \sum_{i=1}^k \lambda_i \xi_i - \sum_{i=1}^k \lambda_i y_i \right\|_{L_1} \\ \le \sum_{i=1}^k |\lambda_i| \|\xi_i - y_i\|_{L_1} \le d_n(\{ \xi_0, \dots, \xi_k \}) \le 2 d_n(\solm)
\end{multline*}
for any space $Y$ of dimension smaller than $n$ realizing the Kolmogorov width of $\solm$. It follows that for the set
\begin{equation*}
  \solm' = \left\{ \sum_{i=1}^k \lambda_i \xi_i : \sum_{i=1}^k |\lambda_1| \le 1 \right\}
\end{equation*}
we have
\begin{equation}
  d_n(\solm') \le 2 d_n(\solm).
  \label{eq:compare-width}
\end{equation}
Next, we show that $\solm'$ is in fact an $L_1$-ball. To this end, note that according to the exact solution \eqref{eq:exact-solution} the functions $\xi_i$ take the value one on disjoint triangles and zero else. The area of the triangles is $h/2$ with $h = (\mu_{\max} - \mu_{\min})/k$, because the time $t$ is bounded between $0 < t < 1$. Thus, we have
\begin{equation*}
  \|\xi_i\|_{L_1} = \frac{h}{2}. 
\end{equation*}
It follows that 
\begin{equation*}
  \left\| \sum_{i=1}^k \lambda_i \xi_i \right\|_{L_1} 
  = \sum_{i=1}^k |\lambda_i| \|\xi_i\|_{L_1}
  = \frac{h}{2} \sum_{i=1}^k |\lambda_i|.
\end{equation*}
so that
\begin{equation*}
  \solm' = \linspan \{ \xi_1, \dots, \xi_k \} \cap B^{h/2}_{L_1}
\end{equation*}
where $B^{h/2}_{L_1}$ is the $L_1$-ball with radius $h/2$. Choosing $k = 2n$, we obtain the Kolmogorov width
\begin{equation*}
  d_n(\solm') = \frac{h}{2},
\end{equation*}
see e.g. \cite{LorentzGolitschekMakovoz1996}. Using \eqref{eq:compare-width} and $h \sim 1/n$ completes the proof of the lower bound.

In order to prove the upper bound, note that $u_\mu - u_{\mu_i}$ is one on a triangular domain and zero else where $u_{\mu_i}$ are the snapshots used in the proof of the lower bound. Calculating the area of the triangle as before, this yields
\begin{equation*}
  \|u_\mu - u_{\mu_i}\|_{L_1} \le \frac{h}{2}
\end{equation*}
for the parameter $\mu_i$ closest to $\mu$. Thus, a piecewise constant approximation by $u_{\mu_0}, \dots, u_{\mu_k}$ yields the upper bound of the proposition.
\end{proof}

\end{appendices}

\bibliographystyle{plain}

\bibliography{lit,local}

\end{document}